\documentclass[10pt,a4paper, reqno]{amsart}
\usepackage{amssymb, paralist, url, enumitem}
\usepackage{amsfonts, amsmath, wasysym}
\usepackage{latexsym}
\usepackage{graphicx, color}
\usepackage{indentfirst}
\usepackage{stmaryrd}

\usepackage{amssymb}
\usepackage{amsmath}
\usepackage{graphicx, color}
\usepackage{indentfirst}

\newcommand{\C}{\mathbb{C}}

\newcommand{\N}{\mathbb{N}} 

\newcommand{\R}{\mathbb{R}}

\newcommand{\<}{\left<}

\newcommand{\ds}{\displaystyle}
\renewcommand{\>}{\right>}

\newcommand{\gorro}{\widetilde}

\reversemarginpar

\usepackage{xargs}                      
\usepackage[pdftex,dvipsnames]{xcolor}  
\usepackage[colorinlistoftodos,prependcaption,textsize=tiny]{todonotes}
\newcommandx{\unsure}[2][1=]{\todo[linecolor=red,backgroundcolor=red!25,bordercolor=red,#1]{#2}}
\newcommandx{\change}[2][1=]{\todo[linecolor=blue,backgroundcolor=blue!25,bordercolor=blue,#1]{#2}}
\newcommandx{\info}[2][1=]{\todo[linecolor=OliveGreen,backgroundcolor=OliveGreen!25,bordercolor=OliveGreen,#1]{#2}}
\newcommandx{\improvement}[2][1=]{\todo[linecolor=Plum,backgroundcolor=Plum!25,bordercolor=Plum,#1]{#2}}
\newcommandx{\thiswillnotshow}[2][1=]{\todo[disable,#1]{#2}}

\newcommand{\mc}[1]{\mathcal{#1}}

\newcommand{\Rm}{{\rm Rm}}

\def\tr{{\rm tr}}

\def\fle{\rightarrow}
\def\parcial#1#2{\fracc{\partial #1}{\partial#2}}

\def\({\left (}
\def \){\right)}
\newtheorem{mainthm}{Theorem}[]
\newtheorem{maincor}[mainthm]{Corollary}
\newtheorem{thm}{Theorem}[section]

\newtheorem*{thm*}{Theorem} 
\newtheorem{cor}[thm]{Corollary}\newtheorem{prop}[thm]{Proposition}
\newtheorem{lem}[thm]{Lemma}

\theoremstyle{definition}

\newtheorem{defin}[thm]{Definition}

\newtheorem*{rems*}{Remarks}
\theoremstyle{remark}
\newtheorem{rem}[thm]{Remark}

\newcommand{\eps}{\ensuremath{\varepsilon}}

\newtheorem{teor}{\hspace{12pt} Theorem}

\newtheorem{lema}[teor]{\hspace{12pt} Lemma}

\newcommand{\bigzero}{\mbox{\normalfont\LARGE  0}}
\newcommand{\rvline}{\hspace*{-\arraycolsep}\vline\hspace*{-\arraycolsep}}

\numberwithin{teor}{section}

\newcommand{\be}{\begin{enumerate}}
\newcommand{\ee}{\end{enumerate}}
\newcommand{\bi}{\begin{itemize}}
\newcommand{\ei}{\end{itemize}}
\newcommand{\bd}{\begin{description}}
\newcommand{\ed}{\end{description}}
\newcommand{\bec}{\begin{equation}}
\newcommand{\eec}{\end{equation}}
\newcommand{\ba}{\begin{array}}
\newcommand{\ea}{\end{array}}
\newcommand{\bt}{\begin{thm}}
\newcommand{\et}{\end{thm}}
\newcommand{\bdem}{\begin{proof}}
\newcommand{\edem}{\end{proof}}
\newcommand{\bl}{\begin{lema}}
\newcommand{\el}{\end{lema}}
\newcommand{\bnp}{\begin{rem}}
\newcommand{\enp}{\end{rem}}
\newcommand{\bde}{\begin{defin}}
\newcommand{\ede}{\end{defin}}
\newcommand{\bnod}{\begin{rem}}
\newcommand{\enod}{\end{rem}}
\newcommand{\bp}{\begin{prop}}
\newcommand{\ep}{\end{prop}}
\newcommand{\bco}{\begin{cor}}
\newcommand{\eco}{\end{cor}}

\newcommand{\nn}{\nonumber}
\newcommand{\lb}{\label}

\newcommand{\I}{{\rm I}}

\newcommand{\G}{\mathsf{G}}

\DeclareMathOperator{\Ric}{Ric}\DeclareMathOperator{\ad}{ad}

\DeclareMathOperator{\scal}{scal}

\DeclareMathOperator{\Or}{O}

\newcommand{\so}{\mathfrak{so}}
\newcommand{\Lg}{\mathfrak{g}}

\DeclareMathOperator{\trace}{tr}
\DeclareMathOperator{\vol}{vol}\DeclareMathOperator{\Ad}{Ad}

\DeclareMathOperator{\id}{id}

 \DeclareMathOperator{\diam}{diam}
\DeclareMathOperator{\rank}{rank}
\DeclareMathOperator{\spann}{span}

\newcommand{\ene}{\end{equation} }

\usepackage{bm}

\usepackage[draft]{optional}

\renewcommand{\(}{\left(}

\renewcommand{\>}{\right>}
\renewcommand{\)}{\right)}
\def\bal{\begin{align}}
\def\eal{\end{align}}

\numberwithin{equation}{section}
\def\be{\begin{equation}}
\def\ee{\end{equation}}

\def\tr{{\rm tr}}

\def\parcial#1#2{\frac{\partial #1}{\partial#2}}

\def\fle{\rightarrow}

\def\ds{\displaystyle}

\begin{document}

\markright{}


\reversemarginpar

\pagestyle{myheadings}

\title{The Ricci flow under almost non-negative curvature conditions}

\author{Richard H.~ Bamler}
\address{Department of Mathematics, UC Berkeley, Berkeley, CA 94720, USA}
\email{rbamler@math.berkeley.edu}

\author{Esther Cabezas-Rivas}
\address{Goethe-Universit\"at Frankfurt, Robert-Mayer-Str.~10, 60325 Frankfurt, Germany}
\email{cabezas-rivas@math.uni-frankfurt.de}

\author{Burkhard Wilking}
\address{University of M\"unster, Einsteinstrasse 62, 48149 M\"unster, Germany}
\email{wilking@math.uni-muenster.de}

\thanks{The first named author is supported in part by the Alfred P. Sloan Foundation through a Sloan Research Fellowship and by the National Science Foundation under Grant No. DMS-1611906. 
\\ \hspace*{3.7mm} The second named author was partially supported by by  the
MINECO (Spain) and FEDER  project MTM2016-77093-P.
\\ \hspace*{3.7mm}
We would like to thank the American Institute of
Mathematics (AIM) for hosting the workshop {\it Geometric flows and Riemannian geometry} which brought the authors together
and at which occasion the initial questions for this project were
discussed.}

\maketitle

\begin{abstract}
We generalize most of the known Ricci flow invariant non-negative curvature conditions to less restrictive negative bounds that remain sufficiently controlled for a short time.

As an illustration of the contents of the paper, we prove that metrics whose curvature operator has eigenvalues greater than $-1$ can be evolved by the Ricci flow for some uniform time such that the eigenvalues of the curvature operator remain greater than $-C$.
Here the time of existence and the constant $C$ only depend on the dimension and the degree of non-collapsedness. We obtain similar generalizations for other invariant curvature conditions, including positive biholomorphic curvature in the K\"ahler case. We also get a local version of the main theorem.
 
As an application of our almost preservation results we deduce a variety of gap and smoothing results of independent interest, including  a classification  for non-collapsed manifolds with almost non-negative curvature operator and a smoothing result for singular spaces coming from sequences of manifolds with lower curvature bounds.
We also obtain a short-time existence result for the Ricci flow on open manifolds with almost non-negative curvature (without requiring upper curvature bounds).

\end{abstract}

\section{Introduction and main results}
The search for invariant curvature conditions has proven to be key to the study of Ricci flows. While this search has been very fruitful, most of the known invariant curvature conditions are rather restrictive, because they impose strong positivity requirements on certain curvature quantities.
For example they entail that the scalar, Ricci or even sectional curvature is positive, which heavily constrains  the topology of the underlying manifold.

In this paper, we show that many of the known invariant curvature conditions can be generalized to curvature bounds that deteriorate under the flow by at most a controlled factor within a short time-interval. 
These bounds only require that certain curvature quantities are bounded from below by a negative constant. Therefore, they hold for any metric after rescaling by a sufficiently large factor and hence we don't impose any topological restrictions.

Our first main result generalizes the invariance of the non-negativity of the curvature operator, which was originally observed by Hamilton (see \cite{Ham86}) and further studied by B\"ohm and the third author (see \cite{BW2}). 
Recall that we cannot expect that negative lower bounds for the curvature operator are in general invariant under the Ricci flow (see e.g.~\cite{Davi} for a counterexample).
The following theorem serves as an illustration for generalizations of a larger class of invariant curvature conditions, as presented in Theorem~\ref{thm:ann_gral} below.

\begin{mainthm}\label{thm:annco} 
Given $n \in \mathbb N$ and a constant $v_0 >0$, there exist 
positive constants $C = C(n, v_0) > 0$ and $\tau = \tau(n, v_0) > 0$  such that the following holds.  

Let $(M^n,g)$ be a complete Riemannian manifold  with bounded curvature satisfying
$$\vol_g\big(B_g(p, 1)\big) \ge v_0 \quad \text{ for all }  p\in M \qquad \text{ and} \quad  \Rm_g\ge -\eps\ge -1,$$
 i.e. the lowest eigenvalue of
$\Rm_g$ is bounded below by $-\eps \in [-1,0]$. Then the Ricci flow $g(t)$ with initial metric $g$ exists until time $\tau$,
and we have the curvature bounds
 $$\Rm_{g(t)}\ge -C \eps \quad \text{
 and } \quad |{\Rm_{g(t)}}| \le \frac{C}{t} \quad \text{ for all } t\in (0,\tau].$$  
\end{mainthm}

Notice that the bound $\Rm_g \geq - \eps$ in the theorem above can be rephrased by saying that the linear combination $\Rm_g + \eps \, \I$, where $\I$ denotes the curvature operator of the unit round $n$-sphere, is non-negative definite.
Using this point of view, we can generalize Theorem~\ref{thm:annco} to further invariant curvature conditions; see the following theorem.
Hereafter we will denote curvature conditions by $\mc C$ and we will write $\Rm_g \in \mc C$ to indicate that $\Rm_g$ satisfies the corresponding curvature condition.

\begin{mainthm} \lb{thm:ann_gral}
Given $n \in \mathbb N$ and a constant $v_0 >0$, there exist 
positive constants $C = C(n, v_0) > 0$ and $\tau = \tau(n, v_0) > 0$  such that the following holds.

Let $(M^n,g)$ be a complete Riemannian manifold  with bounded curvature and consider one of the following curvature conditions $\mc C$:

\begin{enumerate}[label=(\arabic*)]
\item[{\rm (1)}] non-negative curvature operator,
\item[{\rm (2)}] 2-non-negative curvature operator \\
(i.e. the sum of the lowest two eigenvalues is non-negative),
\item[{\rm (3)}] non-negative complex sectional curvature \\
 (i.e. weakly $\text{\rm PIC}_2$, meaning that taking the cartesian product with $\R^2$ produces a non-negative isotropic curvature operator),
\item[{\rm (4)}] weakly $\text{\rm PIC}_1$ \\ (i.e. taking the cartesian product with $\R$ produces a non-negative isotropic curvature operator),
\item[{\rm (5)}] non-negative bisectional curvature, in the case in which $(M,g)$ is K\"ahler with respect to some complex structure $J$.
\end{enumerate}

Assume that
\bec \lb{hyp_thm2}
\vol_g\big(B_g(p, 1)\big) \ge v_0 \quad \text{ for all }  p\in M \qquad \text{ and} \quad  \Rm_g + \eps \I \in \mc C,
\eec
 for some $\eps \in [0, 1]$. Then the Ricci flow $g(t)$ with initial metric $g$ exists until time $\tau$, is K\"ahler if $(M,g)$ is K\"ahler, and we have the curvature bounds
\bec \lb{curv_bdd_thm2}
\Rm_{g(t)} +C \eps  \, \I  \in \mc C\quad \text{
 and } \quad |\Rm_{g(t)}| \le \frac{C}{t} \quad \text{ for all } t\in (0,\tau].
\eec 
\end{mainthm}

We remark that in case (1) we recover Theorem~\ref{thm:annco}.

Theorems~\ref{thm:annco} and \ref{thm:ann_gral} imply a variety of smoothing and gap results.
As a first application we show that volume non-collapsed closed manifolds that satisfy certain almost non-negative curvature conditions also admit metrics that satisfy the corresponding strict condition.

\begin{maincor} \label{anco_nnc}
Given $n \in \mathbb N$ and positive constants $D, v_0$, there exists a constant $\eps = \eps(n, v_0, D) >0$ such that the following holds.

Let $\mc C$ be one of the curvature conditions listed in items $(1)-(5)$ of Theorem \ref{thm:ann_gral}. 
Then any closed Riemannian manifold $(M^n,g)$ with 
$$\diam_g(M) \le D, \qquad  
\vol_g(M) \ge v_0 \quad \text{ and } \qquad \Rm_g + \eps \, \I \in \mc C$$ also admits a metric whose curvature operator lies in $\mc C$. 

(Note that in the case (5), we don't claim that the new metric is K\"ahler with respect to the complex structure for which $(M, g)$ is K\"ahler.)
\end{maincor}

It will be clear from the proof that the metric whose existence is asserted in Corollary~\ref{anco_nnc} is close to the original metric $g$ in the Gromov-Hausdorff sense.

Related to the previous remark, we also obtain the following smoothing result for singular limit spaces of sequences of manifolds with lower curvature bounds.

\begin{maincor} \label{anco_smooth}
Let $\mc C$ be as in Corollary \ref{anco_nnc} and $(X, d_X)$ be the Gromov-Hausdorff limit of a sequence $\{(M_i, g_i)\}_{i=1}^\infty$ of closed Riemannian manifolds satisfying 
\[ 
\vol_{g_i}(M_i) \geq v_0, \qquad \Rm_{g_i} +  \eps_i \, \I \in \mc C, \qquad {\rm diam}_{g_i}(M_i) \leq D.\]
for some sequence $ \{\eps_i\} \subset (0, 1]$ with $\eps_i \to \eps_\infty$, as $i \to \infty$. Then there exists $\tau = \tau(n, v_0) > 0$, a smooth manifold $M_\infty$ and a smooth solution to the Ricci flow $(M_\infty, g_\infty(t))_{t \in (0, \tau)}$  which satisfies $\Rm_{g_\infty(t)} + \eps_\infty \I \in \mc C$ and is coming out of the (possibly singular) space $(X, d_X)$ in the sense that
\bec \lb{init_sing}
\lim_{t \searrow 0} d_{GH}\big((X, d_X), (M_\infty, d_{g_\infty(t)})\big) = 0.
\eec
In particular, for $\eps_\infty = 0$ the limiting $g_\infty(t)$ satisfies
the corresponding non-negative curvature condition $\mc C$ for all $t \in (0, \tau)$. 

Moreover, for any choice of $\eps_\infty$, the space $X$ is homeomorphic to the manifold $M_\infty$ and the Riemannian distance $d_{g_\infty(t)}$ converges uniformly to a distance function $d_0$ on $M_\infty$ as $t \searrow 0$ such that $(M_\infty, d_0)$ is isometric to $(X,d_X)$.
\end{maincor}

By taking convergent sequences of manifolds as above one can generate a large variety of singular spaces that can be smoothed out by the Ricci flow with lower curvature bound.
Unlike in a recent result of Giannotis and Schulze (see \cite{GS}), these singularities do not need to be isolated and conical. On the other hand, the result in \cite{GS} requires no uniform lower bound on the curvature operator, and hence does not follow from ours.

In the case (5), Corollary \ref{anco_smooth} implies a statement that is similar to a result of Gang Liu (cf.~\cite{Liu}) on the structure of limits of spaces whose bisectional curvature is uniformly bounded from below. 
We thank Zhenlei Zhang for pointing out this application to us.

We will also establish a local version of Theorem~\ref{thm:ann_gral} in the case of non-negative curvature operator and non-negative complex sectional curvature.

\begin{mainthm}\label{anco_local} 
Given  $n \in \mathbb N$, $\eps \in [0, 1]$ and $v_0 >0$ there are constants $\tau = \tau(n, v_0, \eps) >0$ and $C = C(n, v_0) > 0$ such that the following holds.

Let $\mc C$ be the curvature conditions listed in item $(1)$ or $(3)$ of Theorem~\ref{thm:ann_gral}.
Let $(M^n,g)$ be any Riemannian manifold (not necessarily complete) and consider an open subset $U \subset M$ and $r >0$ satisfying

\begin{enumerate}[label=(\arabic*)]
\item[{\rm (1)}] The $r$-tubular neighborhood around $U$, $B_r(U)$ is relatively compact in $M$.
\item[{\rm (2)}] $\Rm_g  + \frac{\varepsilon}{r^2} \I \in \mc C$ on  $B_r(U)$.
\item[{\rm (3)}] $\vol_g\big(B_g(x, r)\big) \ge v_0 r^n$ for every $x \in U$
\end{enumerate}
Then there is an (incomplete) Ricci flow $(g(t))_{t \in [0,\tau r^2)}$ on $U$ with initial metric $g$ and we have the curvature bounds
 $$\Rm_{g(t)}(x) + C\cdot \frac{\eps}{r^2} \, \I  \in \mc C \qquad \text{for all} \quad x \in  U, \ t \in [0, \tau r^2)$$
 and $$|\Rm_{g(t)}| \le \frac{C}{t} \quad \text{on} \quad U \qquad \text{ for all } \quad t\in (0,\tau r^2).$$  
\end{mainthm}

By applying this result to a sequence of larger and larger balls, we obtain the following short-time existence result on complete manifolds with possibly unbounded curvature:

\begin{mainthm}\label{ste_new} 
Given $n \in \mathbb N$, $\eps \in [0, 1]$ and a constant $v_0 >0$, there exist 
positive constants $C = C(n, v_0) > 0$ and $\tau = \tau(n, v_0) > 0$  such that the following holds.

Let $\mc C$ be the curvature conditions listed in item $(1)$ or $(3)$ of Theorem~\ref{thm:ann_gral}.
Let $(M^n,g)$ be any complete Riemannian manifold satisfying \eqref{hyp_thm2}.  Then there exists a complete Ricci flow $(M, g(t))_{t \in [0, \tau)}$ with $g(0) = g$ and so that the curvature bounds in \eqref{curv_bdd_thm2} hold.
\end{mainthm}

Theorem~\ref{anco_local} can be used to strengthen a result by Lott (see \cite[Proposition 1]{Lott_Ff}) on the geometry of locally volume collapsed manifolds with almost non-negative curvature operator.
We also mention that in the same reference, Lott asks whether each simply connected manifold with almost non-negative curvature operator is diffeomorphic to  a torus bundle over a compact symmetric space. 
 Corollary~\ref{anco_nnc} gives an affirmative answer to Lott's question in the non-collapsed case.

In dimension 3, Theorem~\ref{thm:ann_gral} and Corollaries~\ref{anco_nnc} and \ref{anco_smooth} were established by Simon in \cite{Simon09, Simon12} for the case of almost non-negative and 2-non-negative curvature operator, which in dimension 3 is equivalent almost non-negative sectional and Ricci curvature, respectively. Theorem~\ref{anco_local}  can be regarded as a higher dimensional version of the results proved by Simon in \cite{Simon13, Simon17} for dimensions 2 and 3, respectively. Note that even for low dimensions our result is new because we do not assume short-time
existence. Theorem \ref{ste_new} is in turn a generalization of the short time existence result in \cite{CRW} by relaxing the corresponding non-negative curvature condition.

Finally, let us explain the main idea of the proof of Theorem~\ref{thm:annco}; Theorem~\ref{thm:ann_gral} will follow similarly, modulo some technical details. 
Denote by $\ell$ the negative part of the smallest eigenvalue of $\Rm$ (see (\ref{def_ell})). In Theorem~\ref{thm:annco} we assume that $\ell \leq 1$ at time $0$.
By standard formulas, $\ell$ roughly satisfies an evolution inequality of the form
\begin{equation} \label{eq_ell_intro}
 \partial_t \ell \leq \Delta \ell + C_1 \scal \cdot \ell + C_2 \ell^2 .
\end{equation}

Traditionally, the invariance of a curvature condition is reduced to a pointwise invariance property via a maximum principle ({\sc ode-pde} comparison).
Unfortunately, this strategy only works for specific curvature conditions. 
Indeed, in dimensions $n \geq 3$, the bound $\Rm \geq \eps$ satisfies this pointwise invariance only if $\eps \geq 0$.
This is why in Theorem~\ref{thm:annco}, we cannot assert strict invariance of the lower bound on the curvature operator.

The degree to which this invariance fails at each point is measured by the reaction term $C_1 \scal \cdot \ell + C_2 \ell^2$ in \eqref{eq_ell_intro}.
Our goal will be to show that this failure is compensated by the diffusion of \eqref{eq_ell_intro}.
In other words, we will bound the influence of the reaction term on $\ell$ in an integral sense.
For example, if we consider the evolution of the integral of $\ell$, then we obtain
\begin{multline} \label{eq_dt_int_ell}
 \frac{d}{dt} \int_M \ell d\mu_t\leq \int_M \big( \Delta \ell + C_1 \scal \cdot \ell + C_2 \ell^2 - \ell \cdot \scal \big) d\mu_t \\
  =  (C_1 -1 )\int_M \scal \cdot \ell d\mu_t + C_2 \!\int_M \ell^2 d\mu_t . 
\end{multline}
Note that the $-\ell \cdot \scal$ term is generated by the distortion of the volume element.
A crucial step in our proof will be to show that we can choose $C_1 = 1$ in \eqref{eq_ell_intro}, which implies that the first term on the right-hand side of (\ref{eq_dt_int_ell}) vanishes.
So as long as $\ell$ remains bounded, its integral cannot grow too fast.
In section~\ref{sec_hk} we will generalize this principle and derive a Gaussian estimate for the heat kernel of the linearization of (\ref{eq_ell_intro}) under certain a priori assumptions.
This estimate will enable us to derive \emph{pointwise} estimates for $\ell$ by localizing (\ref{eq_dt_int_ell}).
Theorem~\ref{thm:annco} will then follow via a continuity argument (see section~\ref{sec_main_proof} for details)

Theorem \ref{anco_local} will follow from Theorem \ref{thm:ann_gral} by a suitable conformal change. Then a limiting argument gives Theorem \ref{ste_new} (see section~\ref{sec_loc_Thm}).

\section{Evolution inequalities for curvature quantities} \lb{algebraic}

\subsection{Background about algebraic curvature operators and preserved curvature conditions}

Consider $\{e_i\}_{i=1}^n$ an orthonormal basis of $\R^n$. Then the set $\{e_i \wedge e_j\}_{i < j}$ forms an orthonomal basis of $\bigwedge^2\R^n$ with respect to the canonical inner product given by
$$\<x \wedge y, z \wedge {\rm v}\> = \<x, z\>\<y, {\rm v}\> - \<x, {\rm v}\>\<y, z\> \qquad \text{for} \quad x,y, z, {\rm v} \in \R^n.$$
We identify $\bigwedge^2\R^n$ with $\so(n, \R)$ via the linear transformation determined by
$$e_i \wedge e_j \longmapsto \big(e_i \wedge e_j\big)_{k l}:= \delta_{ik} \delta_{j l} - \delta_{il} \delta_{jk}.$$

Let $S^2_B(\so(n))$ denote the set of algebraic curvature operators on $\R^n$. Every $\Rm \in  S^2_B(\so(n))$ is the symmetric bilinear form on $\so(n, \R)$ defined by
$$\Rm(e_i \wedge e_j, e_k \wedge e_l) = \Rm_{ijkl},$$
where the left hand side is the corresponding $(4, 0)$ tensor on $\R^n$. We extend each $\Rm \in S^2_B(\so(n))$ complex bi-linearly to a map $\Rm: \so(n,\C) \times \so(n,\C) \fle \C$. Here $\so(n,\C)$ is endowed with the natural Hermitian inner product
$$\<v, w\> = -\tr(v \bar{w})/2,$$
where $v \mapsto \bar v$ denotes complex conjugation. Recall that $\Rm(v, \bar v) \in \R$.

As explained in \cite{WLie}, many of the (known) curvature conditions that are preserved by the Ricci flow can be described by means of convex cones of the form
\begin{equation} \label{eq_def_CC}
 \mc C(S, h) = \big\{ \Rm \in S^2_B(\so(n)) \;\; \big| \;\; \text{$\Rm (v, \bar v) \geq h$ for all $v \in S$} \big\}, 
\end{equation}
where $S \subset \so (n, \C)$ is a subset that is invariant under the natural $SO(n, \C)$-action and $h \in \R$.
We will focus on the cones $\mc C(S) = \mc C(S, 0)$ for the following choices of $S$:

\def\arraystretch{1.3}
\begin{tabular}{|l|r|}
\hline
{\bf choice of $\bm{S = S_i}$} & $\bm{\mc C(S_i) = \{ \text{\bf  ... curvature operators}\}}$ \\
\hline
$S_1 =  \so (n, \C)$ & non-negative   \\
\hline 
$S_2 = \{ v \in \so (n, \C) \, | \, \tr (v^2) =0 \}$ & 2-non-negative  \\
\hline 
$S_3 = \{ v \in \so (n, \C) \, | \, \rank (v) = 2, v^2 = 0 \}$ &  weakly positive isotropic (PIC)  \\ \hline 
 $S_4 = \{ v \in \so (n, \C) \, | \, \rank (v) = 2, v^3 = 0 \}$ & weakly $\text{PIC}_1$  \\ \hline 
$S_5 = \{ v \in \so (n, \C) \, | \, \rank (v) = 2 \}$ & weakly $\text{PIC}_2/$non-negative complex  \\
\hline
\end{tabular}

\def\arraystretch{1}
\medskip

\noindent Preservation of 2-non-negative curvature was originally proved by H.~Chen \cite{Chen}. The invariance of weakly PIC was first showed in dimension four by Hamilton \cite{Ham97}; the general case was obtained independently by S.~Brendle and R.~Schoen \cite{BS} and by H.~T.~Nguyen \cite{Huy}. The invariant conditions weakly $\text{PIC}_1$ and $\text{PIC}_2$ were in turn introduced by Brendle and Schoen in \cite{BS} and play a key role in their proof of the differentiable sphere theorem.

Let $\mc C \subset S^2_B(\so(n))$ be one of the above curvature conditions and denote by $\I \in S^2_B(\so(n))$ the constant curvature operator of scalar curvature $n(n-1)$.
Note that for any $\Rm \in S^2_B(\so(n))$ we have $\Rm + \ell \, \I \in \mc C$ for sufficiently large $\ell$.
So, as $\mc C$ is closed, we can consider the smallest $\ell \geq 0$ for which $\Rm + \ell \, \I \in \mc C$.
Then, for example, in the case $S = S_1$ this $\ell$ is equal to the negative part of the smallest eigenvalue of $\Rm$.
Analogously, for a Ricci flow $(M, g(t))$ we define
\bec \lb{def_ell}
\ell(p, t):= \inf\{\alpha \in [0,\infty) \ | \  \Rm_{g(t)}(p) + \alpha \,  \I  \in \mc C\}.
\eec
The main goal of this section is to derive an evolution inequality for $\ell$ by using the evolution equation of the curvature operator $\Rm$  under the Ricci flow: 
\begin{equation} \label{eq_evolution_Rm_in_proof}
 \nabla_t \Rm = \Delta \Rm + 2 \, Q(\Rm), \qquad \text{where} \qquad Q(\Rm) := \Rm^2+ \Rm^\#,
\end{equation}
where $\Rm^\sharp(u, v):= -\tfrac1{2} \tr({\rm ad}_u \, \Rm \, {\rm \ad}_{v} \, \Rm)$ and $\nabla_t$ denotes the natural space-time extension of the Levi-Civita connection $\nabla^{g(t)}$ so that it is compatible with the metric i.e. $\partial_t |X|^2_{g(t)} = 2\<\nabla_t X, X\>_{g(t)}$.

\subsection{Riemannian case} 

Let $A$ and $B$ be symmetric bilinear forms on $\R^n$. The Kulkarni-Nomizu product $A \owedge B \in S^2_B(\so(n))$ is given by
\bec \lb{def_KN}
  (A \owedge B)_{ijkl} = A_{ik} B_{jl} + A_{jl} B_{ik} - A_{il} B_{jk} - A_{jk} B_{il}. 
\eec

For any $\Rm \in S^2_B(\so(n))$ we will denote by $\Ric = \Ric (\Rm)$ and $\scal = \scal (\Rm)$ the associated Ricci and scalar curvatures. From \cite[Lemma 2.1]{BW2} one can easily compute that
\begin{equation} \label{eq_Rm_I_minus_Rm_Q}
 Q(\Rm + \ell \, \I ) = Q (\Rm) +   \ell \Ric \owedge \id + (n-1) \ell^2 \I,
\end{equation}
where we write $\Ric \owedge \id$ to denote $\Ric \owedge g$ when $g_{ij} = \delta_{ij}$. Note that $\Ric \owedge \id = 2 \Ric \wedge \id$, where $\wedge$ is the wedge-product defined in \cite{BW2}. If we view elements $v \in \so (n, \C)$ as antisymmetric matrices with complex entries and the Ricci tensor $\Ric$ as a symmetric matrix with real entries, then
 \eqref{def_KN} implies
\bec \lb{Ric_id_vv}
(\Ric \owedge \id) (v, \bar v) =  R_{ij} v_{is} \bar v_{js} = -  \tr (\Ric \, v \, \bar v), 
\eec
where the last term denotes the trace of the product of three matrices.
The constant curvature operator $\I$ in turn satisfies 
\bec \lb{Ivv}
\I (v, \bar v) = \tfrac1{2}(\id \owedge \id)(v, \bar v) =  |v|^2.
\eec

 In the following lemma we will derive an evolution inequality for $\ell$ from \eqref{eq_evolution_Rm_in_proof} under an algebraic assumption on $\Ric \owedge \id$, which we will subsequently verify in the proof of the main result of this section, Proposition~\ref{lem:ell}.
\begin{lem} \label{lem_ell_equation}
Let $\mc C = \mc C(S)$ be a cone of the above form. Assume that there is a constant $\lambda \in [0, \infty)$ such that for any $\Rm \in \partial \mc C$ and $v \in S$ with $\Rm (v, \bar v) = 0$ we have
\bec \lb{cross_term} ( \Ric \owedge \id -  \tfrac{\scal}{2} \, \I)(v, \bar v) \leq   \lambda \sqrt{Q(\Rm) (v, \bar v)}. \eec
Then for any Ricci flow $(M, g(t))$ there is a constant $C \in (0, \infty)$ such that 
\begin{equation} \label{eq_ell_lambda_claim}
 \partial_t \ell \le \Delta \ell +  \scal \ell + C \ell^2
\end{equation}
holds in the following barrier sense: for any $(q, \tau) \in M \times (0, T)$
 we can find a neighborhood $\mc U \subset M \times (0, T)$ of $(q, \tau)$
and a $C^\infty$ (lower barrier) function $\phi:\mc U \fle \R$ such that $\phi \leq  \ell$ on $\mc U$, with equality at $(q, \tau)$ and
\begin{equation} \lb{eq_barrier}
 (\partial_t - \Delta) \phi  \leq \scal \ell + C \ell^2 \qquad \text{at } \quad (q, \tau).
\end{equation}
\end{lem}
\noindent Recall that by standard arguments, inequality \eqref{eq_barrier} holds also in the viscosity sense and in the sense of distributions (see e.g. \cite[Appendix]{Mant} for an elliptic version).

 We remark that it will be crucial for the remainder of this paper that the coefficient in front of the $\scal \ell$ term in \eqref{eq_ell_lambda_claim} is equal to $1$.  Observe also that, as $\mc C$ is preserved by the ODE $\frac{d}{d t} \Rm = 2 Q(\Rm)$, we have 
$Q(\Rm) (v, \bar v) \geq 0$ for all $v$ as in the statement, and hence the right hand side of \eqref{cross_term} is well-defined.

\begin{proof}
Set ${\Rm}^\ast(p, t) := \Rm_{g(t)}(p) + \ell (p, t) \I$; by \eqref{def_ell} it is clear that ${\Rm}^\ast \!\in \mc C$. For an arbitrary point $(q, \tau) \in M \times (0, T)$, we can assume that $\Rm^\ast(q, \tau) \in \partial \mc C$. Indeed, if it belongs to the interior of $\mc C$, then $\ell \equiv 0$ in a neighborhood of $(q, \tau)$, and the barrier function $\phi \equiv 0$ will satisfy \eqref{eq_barrier}.

As $S$ is a cone over a compact subset, we can find $v \in S$ with $|v|  =  1$
such that $\Rm^\ast(v, \bar v) = 0$ at $(q, \tau)$. Then $\I(v, \bar v) = 1$ and we are in position to apply \eqref{cross_term} to deduce that at $(q, \tau)$ we have
\[-\lambda \sqrt{Q(\Rm^\ast) (v, \bar v)}  \leq   \tfrac1{2} \scal^\ast -   (\Ric^\ast \owedge \id)(v, \bar v)  =   \tfrac1{2} \scal-   (\Ric \owedge \id)(v, \bar v)  + \tfrac1{2}(n-1)(n-4)\, \ell,\]
where $\Ric^\ast = \Ric(\Rm^\ast)$ and $\scal^*= \scal(\Rm^\ast)$. Multiplying this inequality  by $\ell$ and adding \eqref{eq_Rm_I_minus_Rm_Q} yields
\bec \lb{Q-sqrtQ}
Q(\Rm^\ast)(v, \bar v) - \lambda \ell \sqrt{Q(\Rm^\ast) (v, \bar v)}  \leq Q(\Rm)(v, \bar v) +  \tfrac1{2} \scal \, \ell +  \tfrac1{2} (n-1) (n -2) \ell^2.
\eec

On the other hand, extend $v$ smoothly to a neighborhood $\mc U$ around $(q, \tau)$ in the following way: take $u$ be the real part of $v$, first extend $u$ to a neighborhood of $q$ in $M$ by parallel translation along radial geodesics using $\nabla^{g(t)}$, and then extend $u$ in time to make it constant in time in the sense that $\nabla_t u = 0$, for the space-time connection in \eqref{eq_evolution_Rm_in_proof}. Applying the same extension to the imaginary part of $v$, we get $v(x,t)\in \so(n, \C)$ for all $(x, t) \in \mc U$ of unit norm and so that $\nabla_t v = \Delta v = 0$ at $(q, \tau)$.

Next the function $\phi:= -\Rm(v, \bar v)$ is defined in $\mc U$ and gives a lower barrier for $\ell$ in that neighborhood.  Hence at $(q, \tau)$ we get
using \eqref{Q-sqrtQ}
\begin{align*}
(\partial_t - \Delta) \phi & = -[(\nabla_t - \Delta) \Rm](v, \bar v) = -2 Q(\Rm)(v, \bar v) \\ & \leq   \scal \, \ell + (n-1)(n-2)\ell^2 -2\big( \sqrt{Q(\Rm^\ast)(v, \bar v)} - \tfrac{\lambda}{2} \ell\big)^2 + \tfrac{\lambda^2}{2} \ell^2
\end{align*} 
and the result follows with $C:= \tfrac{\lambda^2}{2} + (n-1)(n-2)$.
\end{proof}

We can finally state the main result of this section.
\begin{prop} \label{lem:ell}
Let $\mc C = \mc C(S) \subset S^2_B(\so(n))$ denote the curvature condition corresponding to $S = S_1, \ldots, S_5$, as defined above, and assume $n\neq 6$ if $S = S_3$. Then there is a constant $C \in (0, \infty)$ such that the following holds:
Let $(M,g(t))$ be a solution of the Ricci flow, and we define $\ell(p,t)$ as the minimal number in (\ref{def_ell}).  Then $\ell$ satisfies \eqref{eq_ell_lambda_claim} in the barrier and viscosity sense. 
 \end{prop}

\begin{proof}
By Lemma \ref{lem_ell_equation}, we need to verify \eqref{cross_term} for each choice of $S$. 
In all cases we will consider $\Rm\in \mc C(S)$ and an element $v =  u + i w  \in S$ with $|v| = 1$ and
\bec \lb{Rmvv=0}
\Rm (v, \bar v) = \Rm (u, u) + \Rm (w, w) = 0, \qquad \text{ where } \quad u, w \in \so(n, \R).
\eec
Moreover, it is easy to check that $\tr(\Ric v \bar v)= \tr(\Ric u^2) + \tr(\Ric w^2)$ and hence by means of \eqref{Ric_id_vv} and \eqref{Ivv}, the desired bound \eqref{cross_term} can be written as
\begin{equation} \lb{eq_4Ric_tr_vv}
 -  \tr \big(\Ric \, (u^2+ w^2)\big) \leq \tfrac1{2}\scal  (|u|^2 + |w|^2) + \lambda \sqrt{Q(\Rm)(v, \bar v)}
\end{equation}

\fbox{$\bm{S= S_1}$} \  
We will prove \eqref{eq_4Ric_tr_vv} for $\lambda = 0$. As $\Rm \geq 0$, by \eqref{Rmvv=0} $u$ and $w$ are in the kernel of $\Rm$. Then it is enough to 
consider the case $v=u$, that is, to show that
$$- \tr(\Ric u^2) \leq \tfrac1{2} \scal \quad \text{for any} \quad u \in \so(n, \R) \quad \text{with} \quad |u|=1 \quad \text{and}\quad \Rm(u, u) = 0.$$

Choose an orthonormal basis $e_1, \ldots, e_n \in \R^n$ such that
 \[u(e_{2i-1}) = a_i e_{2i}, \qquad u( e_{2i} ) = - a_i e_{2i-1} \]
for some real numbers $a_1, a_2, \ldots, a_m \in \R$ for $m \leq \frac{n}2$.
Note that if $n$ is odd, then $u(e_n) = 0$.
We will carry out all the following calculations in the basis $e_1, \ldots, e_n$.
We have $a_1^2 + a_2^2 + \ldots = |u|^2 = 1$ and moreover
\begin{equation} \label{eq_case_1_goal}
- \tr(\Ric u^2)= \sum_{i = 1}^m a^2_i ( R_{(2i)(2i)} + R_{(2i-1)(2i-1)}  ). 
\end{equation}

 If we express the sectional curvature in the $e_i \wedge e_j$ direction by $K_{i,j} := R_{ijij}$, we get
\begin{equation} \label{eq_Rm_v_v_0}
 0 =  \Rm (u, u) = \sum_{i =1}^m a_i^2 K_{2i-1, 2i} + 2 \sum_{1 \leq i < j \leq m} a_i a_j R_{(2i -1) 2i (2j -1) 2j}. 
\end{equation}
Next, for any $1 \leq i < j \leq m$ and
$ \omega_{ij} := a_j (e_{2i-1} \wedge e_{2i} ) + a_i (e_{2j-1}  \wedge e_{2j} ),$
  we have
\begin{equation} \label{eq_Rm_ai_aj}
 0 \leq \Rm (\omega_{ij}, \omega_{ij} ) 
= a^2_j K_{2i-1, 2i} + a_i^2 K_{2j-1, 2j} + 2\, a_i a_j R_{(2i-1)2i(2j-1)2j}. 
\end{equation}
Adding (\ref{eq_Rm_ai_aj}) for all $1 \leq i < j \leq m$ and subtracting (\ref{eq_Rm_v_v_0})  yields
$$\sum_{i = 1}^m a_i^2 K_{2i-1, 2i} \leq 
  \sum_{1 \leq i < j \leq m} \big( a^2_j K_{2i-1, 2i} + a_i^2 K_{2j -1,  2j} \big) 
 \leq  \sum_{i =1}^m  (1 - a_i^2) K_{2i-1 , 2i}.$$
So
\[ 2\sum_{i =1}^m a_i^2 K_{2i -1 ,2 i} \leq  \sum_{i =1}^m   K_{2i -1 , 2i}, \]
and we get by \eqref{eq_case_1_goal} that
\begin{align*}
- \tr(\Ric u ^2) &= \sum_{i = 1}^m a_i^2 \Big( 2 K_{2i-1, 2i} + \sum_{\substack{1 \leq j \leq 2m \\ j \neq 2i -1, 2i}} (K_{2i, j} + K_{2i-1, j}) \Big) \notag \\
&\leq   \sum_{i = 1}^m   K_{2i-1, 2i} + \!\!\! \sum_{1 \leq i < j \leq m} \!\!\!\!(a_i^2 + a_j^2 ) (K_{2i, 2j} + K_{2i-1, 2j} + K_{2i, 2j-1} + K_{2i-1, 2j-1}). 
\end{align*}
\noindent Finally $a_i^2 + a_j^2 \leq 1$ implies $- \tr(\Ric u^2)  \leq \frac1{2} \scal$, as we wanted to show.

\medskip

\fbox{$\bm{S= S_2}$} \ 
Take $v \in S_2$ as in \eqref{Rmvv=0}. From $0 = \tr (v^2) = \tr (u^2 - w^2 ) + 2 i \tr (u w) = 2(|w|^2 - |u|^2)   - 4 i \langle u, w \rangle$,
we get $|u|^2 = |w|^2 = \frac12$ and $\<u, w\> = 0$. By the discussion in \cite[section 3]{WLie} we also have 
\bec \lb{QgeqRm2}
Q(\Rm)(v, \bar v) \geq \Rm^2(v, \bar v).
\eec

Let $\mu_1 \leq \mu_2 \leq \ldots \leq \mu_{n(n-1)/2}$ be the eigenvalues of $\Rm$.
The curvature assumption implies $\mu_1 + \mu_2 \geq 0$ and $\mu_i \geq 0$ for all $i \geq 2$. So the eigenvalues of $\Rm^2$ are $\mu_1^2 \leq \mu_i^2$ for all $i \geq 2$ and hence $\mu_1^2 \leq \Rm^2 (v, \bar v) \leq Q(\Rm)(v, \bar v)$, which follows from \eqref{QgeqRm2}. Then 
\begin{equation} \label{eq_lower_Rm_bound}
 \Rm (\eta, \bar\eta) \geq \mu_1 |\eta|^2 \geq - \sqrt{Q(\Rm)(v, \bar v)} \cdot |\eta|^2 \qquad  \text{for any } \quad {\eta} \in \so (n, \C). 
\end{equation}
This inequality and \eqref{Rmvv=0} imply the bound
\bec \lb{sust_Rm0}
 | \Rm (u, u) | = | \Rm (w, w) | \leq \tfrac1{2}\sqrt{Q(\Rm)(v, \bar v)}.
\eec

We now carry out a similar analysis as in the first case for each $u$ and $v$.
There are, however, two key differences.
First, we need to use the bound (\ref{eq_lower_Rm_bound}) in lieu of the non-negativity of $\Rm$ in \eqref{eq_Rm_ai_aj}. 
And second, the identity $\Rm (u, u) = 0$ in \eqref{eq_Rm_v_v_0} has to be substituted by \eqref{sust_Rm0}. Note that for the last line of the argument we will use that 2-non-negative curvature operator implies $K_{i,j} + K_{i,j+1} \geq 0$ for all $i \neq j$. Taking these modifications into account, we obtain \eqref{eq_4Ric_tr_vv} for some $\lambda > 0$.

\medskip

\fbox{ \bm{$S=S_4, \,S_5$}} \
Recall that $\rank (v) = 2$ and $v = u + i w$.
So $\rank (u) , \rank (w) \leq 4$.
As $u$ is skew-symmetric, it has eigenvalues of the form $ia, -ia, ib, -ib, 0, \ldots, 0$, $|a| \geq |b|$. 
By multiplying $v = u + i w$ with a complex number of norm 1, we may assume additionally that $|a|$ is maximized.
We can now choose an orthonormal basis $e_1, \ldots, e_n$ of $\R^n$ such that
\[
u = \begin{pmatrix}
  \begin{matrix}
  0 & a & 0 & 0  \\-a & 0 & 0 & 0  \\ 0 & 0 & 0 & b  \\ 0 & 0 & -b & 0
  \end{matrix}
  & \rvline &  \bigzero  \\
\hline 
 \parbox[c]{0.4cm}{\bigzero} & \rvline &
  \begin{matrix}
  \parbox[c]{0.4cm}{\bigzero} 
  \end{matrix}
\end{pmatrix}, \qquad
w = \begin{pmatrix}
  \begin{matrix}
  0 & w_{12} & c & 0  \\ -w_{12} & 0 & 0 & d  \\-c & 0 & 0 & w_{34}  \\ 0 & -d & -w_{34} & 0
  \end{matrix}
  & \rvline &  \bigzero  \\
\hline 
 \parbox[c]{0.4cm}{\bigzero} & \rvline &
  \begin{matrix}
  \parbox[c]{0.4cm}{\bigzero} 
  \end{matrix}
\end{pmatrix}
\]
Due to the maximal choice of $|a|$, we find that $w_{12} = 0$.
Moreover, if $a = 0$, then $b = 0$ and again by the maximal choice of $a$ we have $w_{34} = 0$.
On the other hand, if $a \neq 0$, then due to the fact that $\rank (v) = 2$, we must have $w_{34} = 0$ as well.
It follows that
\begin{equation} \label{eq_std_matrix_form}
v = \begin{pmatrix}
  \begin{matrix}
  0 & a & ic & 0  \\ -a & 0 & 0 & id  \\ -ic & 0 & 0 & b  \\ 0 & -id & -b & 0
  \end{matrix}
  & \rvline &  \bigzero  \\
\hline 
 \parbox[c]{0.4cm}{\bigzero} & \rvline &
  \begin{matrix}
  \parbox[c]{0.4cm}{\bigzero} 
  \end{matrix}
\end{pmatrix}, 
\end{equation}
and therefore
\[
v \bar v = u^2 + w^2 = {\rm diag}\big(a^2 + c^2, a^2+d^2, b^2 + c^2, b^2 + d^2, 0, \ldots, 0\big).\]
So $|v|^2 =  1$ amounts to $a^2 + b^2 + c^2 + d^2 = 1$.
Let $\alpha_1,\ldots, \alpha_4\le 1$ denote the first four diagonal entries of the above matrix.
Using as above the notation for the standard sectional curvature planes from above and  we get
\begin{align} \nonumber
-\trace(\Ric v\bar v)&=
 (a^2-b^2)(K_{12}-K_{34}) + (c^2-d^2)(K_{13}-K_{24}) + \!\!\sum_{1\le i<j\le 4 }K_{ij}+\sum_{i=1}^4 \alpha_i \sum_{j=5}^{n}K_{ij}\\
\label{ineq:csec} 
&\le \tfrac1{2}\scal+ (a^2-b^2)(K_{12}-K_{34}) + (c^2-d^2)(K_{13}-K_{24}),
\end{align}
where we have used that $K_{ij}+K_{kj}\ge 0$ holds for all pairwise different $i$, $j$ and $k$.
Next,
\begin{align*}
0&= \Rm(v,\bar v)=\Rm(u,u)+\Rm(w,w)\\
&= a^2K_{12}+b^2 K_{34}+c^2 K_{13}+d^2 K_{24}+2ab \, R_{1324}+2cd\, R_{1324}
\end{align*}
If we switch the roles of $a$ and $b$ and simultaneously the roles of $c$ and $d$ in the above expression, we get the curvature $\Rm(z,\bar z)$ of  some element $z\in S$. 
 Hence
 $$
0\le\Rm(z,\bar z)-\Rm(v,\bar v)=(b^2-a^2)(K_{12}-K_{34}) + (d^2-c^2)(K_{13}-K_{24})
$$
Adding this inequality to \eqref{ineq:csec} yields the desired inequality.

\medskip

\fbox{$\bm{S=S_3}$ with $n\neq 6$} \ 
By the discussion in the previous case we can again choose an orthonormal basis $e_1, \ldots, e_n$ of $\R^n$ such that (\ref{eq_std_matrix_form}) holds.
The condition $v^2 = 0$ implies that $a^2 = b^2 = c^2 = d^2 = \frac14$ and $ad + bc = ac + bd = 0$.
So
$$ - 2 \tr (\Ric v \bar v)  = 
 R_{11} + R_{22} + R_{33} + R_{44} \leq \lambda_{n-3}+\lambda_{n-2}+\lambda_{n-1}+\lambda_n,$$
where $\lambda_1\le \cdots\le \lambda_n$ denote the eigenvalues of $\Ric$.
If $n=4$, we are done and if $n\ge 7$, then the left hand side is bounded from above by $ \scal$ provided we can establish:

{\it \underline{Claim.} For $n \geq 7$, any $\Rm \in \mc C(S_3)$ has 3-non-negative Ricci curvature.}
 
Take a basis $e_1,\ldots,e_n$ of eigenvectors  of $\Ric$. 
We consider the subgroup $\G=\Or(3)\cdot \Or(n-3)\subset \Or(n)$ leaving 
${\rm span}(e_1, e_2, e_3)$ invariant.  Without loss of generality, we can assume that $\Rm$ is fixed by $\G$. In fact, we can replace 
$\Rm$ by the center of mass $\Rm_C$ of its $\G$-orbit. Of course, $\Rm_C$ still has weakly {\rm PIC} and the quantity $\lambda_1 + \lambda_2 + \lambda_3$ remains unchanged if we replace $\Ric$ by $\Ric(\Rm_C)$.

Then we may assume that $\Rm$ is a multiple of the identity on each of the three 
$\Ad_\G$-invariant subspaces $\so(3)$, $\so(n-3)$ and $\Lg^\perp:=(\so(3)\oplus \so(n-3))^\perp$.
Denote by $\mu$ and $\eta$ the eigenvalues of $\Rm$ on $\so(3)$ and on 
$\Lg^\perp$, respectively. 
Since $\Rm$  has weakly {\rm PIC} and we can 
find nilpotent rank 2 matrices in the complexification of $\Lg^\perp$,
we have $\eta\ge 0$. 
Furthermore, we know that $3\mu+3\eta \ge 0$, because this number corresponds to the trace  of $\Rm$ restricted to the upper $4\times 4$-block. 
It now follows that 
\[
 \lambda_1+\lambda_2+\lambda_3=6\mu+(n-3)\cdot 3\eta\ge 0.
\]

It remains to consider the case $n=5$. 
Here we use the fact that $\Rm(v,\bar v)=0$.
Consider again an orthonormal basis $e_1, \ldots, e_5$ or $\R^5$ such that $v$ takes the form (\ref{eq_std_matrix_form}) and recall that $a^2 = b^2 = c^2 = d^2 = \frac14$ and $ad + bc = ac + bd = 0$.
So $a,b,c,d = \pm \frac12$, where we have to choose $+ \frac12$ an odd number of times.
So after possibly permuting $e_1, \ldots, e_4$ and multiplying $v$ by a complex number of norm 1, we can assume
 \[
v=\frac{1}{2}\left(\begin{array}{ccccc}0& 1 & i & 0 &0\\ - 1 & 0 & 0 & -i & 0 \\ -i & 0 & 0 & 1 & 0 \\ 0 & i & - 1 & 0 & 0 \\ 0 & 0 & 0 & 0 & 0
\end{array}\right)
\]
By a second variation argument we have
 $\Rm(\ad_v w,\ad_{\bar v }\bar{w})\ge 0$ for all $w\in\so(n,\C)$, see \cite{WLie}. 
This in turn implies that $\Rm(z_i,\bar z_i)\ge 0$ ($i=1,2$) holds 
for $z_1=(e_1 + ie_4)\wedge e_5$ and $z_2=(e_2 + ie_3)\wedge e_5$. 
Thus $\Ric(e_5,e_5)=\Rm(z_1,\bar z_1)+\Rm(z_2,\bar z_2)\ge 0$ and the claim follows as well.
\end{proof}

\subsection{K\"ahler case}
Assume for the rest of this subsection that $(M, g, J)$ is a K\"ahler manifold of complex dimension $n = 2m$. Recall that the splitting $T_{\C}M = T^{1, 0} M \otimes T^{0,1} M$ into the $\pm i$-eigenspaces of $J$ implies $R(x, y, z, {\rm w}) = 0$ if $x, y \in T^{1, 0} M$ or $x, y \in T^{0, 1} M$ and similarly for $z, {\rm w}$. Now let $\{e_1, \ldots, e_{2m}\}$  be a local  orthonormal frame for $M$ so that $J e_k = e_{m + k}$ for $k = 1, \ldots, m$. We consider the unitary frame $\{z_k\}_{k = 1}^m$ given by $z_k:= \frac1{\sqrt{2}}(e_k - i Je_k) \in T^{1, 0} M$ so that $g_{a \bar b} = g(z_a, \bar z_b) = \delta_{a b}$. Then we can write
$$R_{a \bar b c \bar d}:=\Rm(z_a, \bar z_b, z_c, \bar z_d), \quad R_{a \bar b} = \Ric(z_a, \bar z_b) = - \sum_{c=1}^m R_{a \bar b c \bar c}, \quad \scal = 2 \sum_{a = 1}^n R_{a \bar a},$$
where $g$, $\Rm$ and $\Ric$ are the complex-linear extensions of the corresponding tensors.

We say that $M$ has non-negative bisectional curvature ($K_{\C} = K_\C(\Rm) \geq 0$) if
$$K_{\C}(x, y) :=-R_{a \bar b c \bar d} \, x^a \bar x^b y^c \bar y^d \geq 0 \qquad \text{for all}\quad  x, y \in T^{1,0}_p M \quad \text{and all} \quad p \in M.$$
This condition is known to be preserved under the K\"ahler Ricci flow $\parcial{g_{a \bar b}}{t} = -2 R_{a \bar b}$ (cf.~\cite{Mok} and \cite{Shi} for closed and complete manifolds with bounded curvature, respectively).

Denote by $\tilde \I$ the curvature tensor of $\C {\rm P}^m$, normalized so that it has constant holomorphic sectional curvature $2$. With respect to the unitary frame, we have
$$- \tilde \I_{a \bar b c \bar d} = \delta_{a b} \delta_{cd} + \delta_{ad}\delta_{bc}.$$
For any point $p \in M$ and time $t$, we set, similarly to \eqref{def_ell},
\begin{equation} \label{def_tilde_ell}
 \tilde \ell(p, t):= \inf \big\{\alpha \geq 0 \ | \  K_\C(\Rm_{g(t)} + \alpha \, \tilde \I)(x, y) \geq 0 \; \text{for all }  x, y \in T^{1,0}_p M\big\}. 
\end{equation}

Hereafter we assume that repeated indices are summed over $1, \ldots, m$. 
The quadratic term in the evolution equation \eqref{eq_evolution_Rm_in_proof} under the K\"ahler Ricci flow becomes (see e.g.~\cite[(125) in section 5]{Shi}) 
\[ Q(\Rm)_{a \bar b c \bar d} = - R_{a \bar b r \bar s} R_{s \bar r c \bar d}  - R_{a \bar d  r \bar s} R_{c \bar b s \bar r} +   R_{a \bar s c \bar r} R_{s \bar b r \bar d}
.\]
And the analogue of the Kulkarni-Nomizu product in this setting is  
$$(\Ric \tilde\owedge \id)_{a\bar b c \bar d} := R_{a \bar b} \delta_{cd}  + R_{c \bar d} \delta_{ab} + R_{c \bar b} \delta_{ad} + R_{a \bar d} \delta_{bc}.$$
Now one can easily compute
$$Q(\Rm + \tilde \ell \tilde I) - Q(\Rm) - \tilde \ell^2 Q(\tilde \I) 
 = \tilde \ell \Ric \tilde \owedge \id.$$
So the analogue of Lemma~\ref{lem_ell_equation} is:

\begin{lem}
Assume that there is a constant $\lambda \in [0, \infty)$ such that for every K\"ahler curvature tensor $\Rm$ with $K_\C \geq 0$ and any $x, y \in T^{1,0} M$  for which $K_\C(x,  y) = 0$ we have
\bec \lb{cross_term_Kaehler_lemma} \big( { \Ric \tilde\owedge \id +  \tfrac{\scal}{2} \, \tilde \I}\,\big)(x, \bar x, y, \bar y) \leq  \lambda \sqrt{-Q(\Rm)(x, \bar x, y, \bar y)}. \eec
Then there is a constant $C$ such that in the barrier and viscosity sense
\begin{equation} \label{eq_ell_lambda_claim_2}
 \partial_t \tilde \ell \le \Delta \tilde \ell +  \scal \tilde \ell + C \tilde \ell^2.
\end{equation}
\end{lem}

Using this lemma, we can prove the following analogue of Proposition~\ref{lem_ell_equation}.

\begin{prop} \label{lem:ell_Kaehler}
There is a constant $C < \infty$ such that the following holds:
Let $(M,g(t))$ be a solution of the K\"ahler Ricci flow. Then the minimal number $\tilde \ell(p,t)$ defined by  \eqref{def_tilde_ell} satisfies the evolution inequality \eqref{eq_ell_lambda_claim_2} in the barrier and viscosity sense. 
\end{prop}

\begin{proof}
It is enough to verify (\ref{cross_term_Kaehler_lemma}) for all $\Rm$ with $K_\C \geq 0$ and for  two $(1, 0)$-vectors $x, y$ for which $K_\C(x, y) = 0$. By rescaling and applying a unitary transformation, we may assume that $|x| = |y| = 1$, $x^1 = 1$, $x^a =  0$ for all $a \geq 2$, $y^b  = 0$ for all $b \geq 3$.
Then
\begin{align*}
 \big( { \Ric \tilde\owedge \id +  \tfrac{\scal}{2} \,\tilde \I }\big)(x, \bar x, y, \bar y)  & = 
  R_{1 \bar 1} +  R_{a \bar b}  y^a \bar y^b +  R_{a \bar 1} y^a \bar y^1 +  R_{1 \bar b} \bar y^b  y^1  - \tfrac{\scal}{2} (1+|y^1|^2)
	\nn \\ &
\leq R_{1 \bar 1}(1 + 3 |y^1|^2) + R_{2 \bar 2} |y^2|^2 + 2\big(R_{2 \bar 1} y^2 \bar y^1 +  R_{1 \bar 2} \bar y^2  y^1) -\tfrac{\scal}{2} \nn \\ & 
\leq 4 R_{1 \bar 1} y^1 \bar y^1 + 4\,\Re(R_{1 \bar 2} \bar y^2  y^1) \leq     4 | R_{1\bar1 } y^1 \bar y^1 + R_{1 \bar 2 } y^1 \bar y^2 |,
	\end{align*}
For the second line we used $R_{a \bar a} \geq 0$, $|y^2| \leq 1$ and $R_{1 \bar 1} + R_{2 \bar 2} \leq \frac{\scal}{2}$.
By the discussion in \cite{Mok} we get (see also \cite[Claim 2.2 in Theorem 5.2.10]{IntroK})
\begin{align*}
-Q(\Rm)(x, \bar x, y, \bar y) & \geq \!\sum_{r, s = 1}^m \!|R(x, \bar y, z_r, \bar z_s)|^2  =   R_{1 \bar b r \bar s} R_{s \bar r c \bar 1} \bar y^b y^c   = \sum_{r, s = 1}^m| \!R_{1\bar1 r \bar s} \bar y^1 + R_{1 \bar 2 r \bar s} \bar y^2 |^2 \\ & \geq |R_{1\bar1 s \bar s} \bar y^1 + R_{1 \bar 2 s \bar s} \bar y^2 |^2 \geq \tfrac1{m} | R_{1\bar1 } \bar y^1 + R_{1 \bar 2 }  \bar y^2 |^2
\end{align*}
Thus by using $|y^1| \leq 1$, we obtain the inequality 
\[\sqrt{Q(\Rm)(x, \bar x, y, \bar y)} \geq  \tfrac1{\sqrt{m}}| R_{1\bar1 } y^1 \bar y^1 + R_{1 \bar 2 } y^1 \bar y^2 |,\]
which yields \eqref{cross_term_Kaehler_lemma} for $\lambda \geq 4 \sqrt{m}$.
\end{proof}

\section{Heat kernel estimates for Ricci flows} \label{sec_hk}
Let $(M^n, g(t))$, $t \in [0, T]$, be a complete Ricci flow.
Hereafter we denote by $G(x, t; y, s)$, with $x, y \in M$, $0 \leq s < t \leq T$, the heat kernel
corresponding to the backwards heat equation coupled with the Ricci flow.
 This means that for any fixed $(x, t) \in M \times [0, T]$ we have
\bec \lb{heat_ker_def}
\big(\partial_s + \Delta_{y,s}\big) G(x,t; \,\cdot\,, \,\cdot\,) = 0 \qquad \text{and }
 \qquad \lim_{s \nearrow t}  G(x, t; \, \cdot, s) = \delta_{x}.
\eec
Then for any fixed $(y, s) \in M \times [0, T]$ one can compute that $G(\,\cdot\,, \,\cdot\,;
 y, s)$ is the heat kernel associated to the conjugate equation
\bec \lb{G_conj}
\big(\partial_t- \Delta_{x, t} - \scal_{g(t)} \big)G(\,\cdot\,, \,\cdot\,; y, s) = 0 \quad \text{and }
 \quad \lim_{t \searrow s} G(\, \cdot \,,t, y, s) = \delta_{y}.
\eec
Note that  in the literature it is more common to consider  the fundamental solution of the conjugate heat equation $\partial_t u +\Delta_{x, t} u -\scal_{g(t)} u = 0$. Hereafter $d_t$ and $d\mu_t$ will denote the Riemannian distance and the volume element, respectively, for the metric $g(t)$.

The goal of this section is to obtain Gaussian upper bounds for $G$. A crucial  fact for the proof is that the $L^1$-norm of $G(\cdot, t; y,s)$ is preserved under (\ref{G_conj}), that is,
\begin{equation} \label{eq_int_G_1}
 \int_M G(\cdot\, ,t;y,s) d\mu_t = 1 \qquad \text{for any}  \quad 0 \leq s < t \leq T.
\end{equation}
In the compact case, this follows from the following simple computation: fix some point $y \in M$ and time $s \in [0,T]$ and note that for any $t > s$ we have
\begin{multline*}
 \frac{d}{dt} \int_M G(\cdot\,,t;y,s) d\mu_t 
  = \int_M \Big( \big( \Delta_{x,t} + \scal_{g(t)} \big) G(\cdot\,, t; y,s) - G(\cdot\,,t; y,s) \scal_{g(t)}  \Big) d\mu_t 
  = 0.
\end{multline*}
The general case follows using an exhaustion and limiting argument (see e.g.~\cite[Cor.~26.15]{Chow3}).

\bp \label{thm:heat}
For any $A > 0$, there is a constant $C = C(n, A) < \infty$ such that the following holds: Let $(M^n, g(t))$, $t \in [0, T]$, be a complete Ricci flow  satisfying 
\bec \lb{RF curv assum A/t}
   |{\rm Rm}_{g(t)}| < \frac{A}{t} \quad  \text{and} \quad \vol_{g(t)}\(B_{g(t)}(x,\sqrt{t})\) >
 \frac{t^{n/2}}{A}
 \eec
 for all $(x,t) \in M \times (0, T]$. Then
\[G(x,t;y,s) < \frac{C }{t^{n/2} } \exp \bigg({ - \frac{d^2_s(x,y)}{C t} }\bigg)
\qquad \text{for all} \qquad  0 \leq 2s \leq t \leq T.\]
\ep

\begin{proof}
Parabolic rescaling and application of a time-shift reduces the proof to the case in which $s = 0$ and $t = 1$.
Note that in this process the right-hand side of the second bound in (\ref{RF curv assum A/t}) may change by a controlled factor due to a volume comparison argument.
So, in summary, our goal will be to show that
\begin{equation} \label{eq_normalized_hk_bound}
 G(x,1; y,0) < C \exp \bigg({ - \frac{d_0^2 (x,y)}{C} }\bigg). 
\end{equation}

Take a subdivision $ \{[t_{k+1}, t_{k}]\}_{k \in \N \cup \{0\}}$ of the interval $(0,1]$ with $t_k := 16^{-k}$.
We first bound the heat kernel $G$ restricted to time-slabs of the form $M \times [t_{k+1}, t_k]$.
Notice that (\ref{RF curv assum A/t}) provides curvature and volume bounds, which are uniform in $k$ after a parabolic rescaling that normalizes the size of the time-interval $[t_{k+1}, t_{k}]$. Then by  \cite[Corollary 26.26]{Chow3} we find $C_0 = C_0 (n, A) < \infty$ such that for any $x, y \in M$ and $k \in \N \cup \{0\}$ we have
\begin{equation} \label{eq_std_hk_bound}
 G(x,t_{k} ; y, t_{k+1} ) \leq \frac{C_0}{(t_{k} - t_{k+1})^{n/2}} \exp \bigg( { - \frac{d^2_{\ast} (x, y)}{C_0 (t_{k} - t_{k + 1})} }\bigg),
\end{equation}
where for $d_\ast$ we can use either $d_{t_{k}}$ or $d_{t_{k+1}}$ due to a basic distance distortion argument on $[t_{k +1}, t_{k}]$. Fix a point $x \in M$.
By (\ref{eq_std_hk_bound}) for $d_\ast = d_{t_{k+1}}$ and $k = 0$  we get in particular
\begin{equation} \label{eq_k_1_hk_bound}
 G(x, 1; y, t_1) \leq  2^n C_0 \exp \bigg( { - \frac{d^2_{t_1} (x, y)}{C_0} }\bigg) \qquad \text{for any } \quad y \in M. 
\end{equation}
So by the maximum principle applied to the solution $G(x, 1; \cdot, \cdot)$ to (\ref{heat_ker_def}) we have
\begin{equation} \label{eq_unform_hk_bound}
 G(x, 1; \cdot, \cdot) \leq 2^n C_0 \qquad \text{on} \qquad M \times [0, t_1]. 
\end{equation}
This implies the desired bound \eqref{eq_normalized_hk_bound} if $d_{0} (x, y)$ is controlled. So it remains to estimate $G(x, 1; y, 0)$ whenever $d_{0} (x, y)$ is large.
For this purpose let $d > 0$ be an arbitrary constant and set for each $k \in \N$
\[ a_k := \sup_{M \setminus B_{g(0)} (x, r_k)} G(x, 1; \cdot, t_k) \qquad \text{with } \qquad r_k:= 4 d \cdot(1- 2^{-k}) \]
As $G(x, 1; \cdot, t)$ is continuous up to time $0$ we have
\[ \sup_{M \setminus B_{g(0)} (x, 4d)} G(x, 1; \cdot, 0) = \lim_{k \to \infty} a_{k+1}. \]
Therefore the statement follows if we prove the following:

{\it Claim.} For some $C = C(n, A) < \infty$, we have $a_{k+1} \leq C \exp \ds \bigg( { - \frac{d^2}{C} }\bigg).$

\noindent It is enough to prove the claim for $d$ larger than some constant $\gorro C(n, A)$,  to be specified later. Indeed, for $d \leq \gorro C(n, A)$ we get (\ref{eq_normalized_hk_bound})   from (\ref{eq_unform_hk_bound}) after possibly adjusting $C$.

We will now iteratively bound the numbers $a_k$. For any $y \in M \setminus B_{g(0)} (x, r_{k+1})$ the reproduction formula for $G$ yields
\bec \lb{repr_G}
G(x, 1; y, t_{k + 1}) = \int_M G(x, 1; z, t_{k}) \, G(z,t_{k}; y, t_{k+1}) \, d\mu_{t_{k}}(z) =: \mc I[M].
\eec
We split the integral $\mc I[M]$ into  integrals over $B_k:= B_{g(t_k)}(y, 2^{-k}  d)$ and  over $M \setminus B_k$. In order to estimate $\mc I\big[B_k]$, we first bound $d_{t_k}$ in terms of $d_0$. By Hamilton's distance distortion bound (see \cite[Theorem 17.2]{HamFS} and Editor's note 24 in \cite{Coll}) and (\ref{RF curv assum A/t}) there is a constant $\Lambda = \Lambda(n) < \infty$ such that for any $z \in B_{g(0)} (x, r_{k})$ we have
\begin{align} \lb{Ham_dist}
 d_{t_{k}} (y,z) & \geq d_0 (y,z) - \Lambda \int_0^{t_{k}} \!\!\sqrt{\frac{A}{t}} \, dt \geq (r_{k+1} - r_{k}) - 2 \, \Lambda\, \sqrt{A t_{k}} 
 \geq 2   (d- \Lambda  \sqrt{A}) 2^{-k} \nn \\
 & \geq   2^{-k} d, \qquad \text{as long as } \quad d \geq 2 \, \Lambda \, \sqrt{A}. 
\end{align}
It follows that $B_k \subset M \setminus B_{g(0)} (x, r_{k})$, and hence 
\bec \lb{I_Omega}
\mc I\big[B_k] \leq \int_{M \setminus B_{g(0)} (x, r_{k})} \!\! G(x, 1; \, \cdot \,, t_{k}) \, G(\, \cdot \, ,t_{k}; y, t_{k+1}) \, d\mu_{t_{k}} \leq a_{k},
\eec
where we used the definition of $a_k$ and \eqref{eq_int_G_1}. On the other hand, by \eqref{eq_unform_hk_bound} we get
$$\mc I\big[M \setminus B_k\big] \leq 2^n C_0 \int_{M \setminus B_k}   G(\, \cdot \,,t_{k}; y, t_{k+1}) d\mu_{t_{k}} \leq  C_1 \exp \bigg({ - \frac{ 4^{-k} d^2}{C_1 (t_{k} - t_{k + 1})}} \bigg)$$
for some constant $C_1 = C_1(n, A) < \infty$. The last inequality follows easily by integrating \eqref{eq_std_hk_bound} with $d_\ast = d_{t_k}$ and  volume comparison. Substituting the latter estimate and \eqref{I_Omega} in \eqref{repr_G}, we obtain
\begin{align} \lb{est_ak1}
 a_{k+1} &\leq a_{k} + \, \exp \bigg({ - \frac{d^2 \cdot 4^{-k}}{C_1 \cdot \frac{15}{16} \cdot 16^{-k}} }\bigg) \leq a_1 +  C_1 \sum_{i = 1}^{k} \exp \bigg({- \frac{d^2}{C_1} \cdot 4^i}\bigg) \nn
\\ & \leq 2^n C_0 \exp\bigg({-\frac{d_{t_1}^2(x,y)}{C_0}}\bigg) + C_1 \exp \big({- \frac{d^2}{C_1} }\big)  \sum_{i = 1}^k \exp \bigg({- \frac{d^2}{C_1} \cdot (4^i-1)}\bigg),
\end{align}
where $y \in M \setminus B_{g(0)}(x, 2d)$ and we have used \eqref{eq_k_1_hk_bound} to estimate $a_1$. Now, arguing as in \eqref{Ham_dist}, we deduce
$$d_{t_1}(x, y) \geq d_0(x, y) - \int_0^1 \sqrt{\tfrac{A}{t}}\, dt \geq 2 (d -  \Lambda \sqrt{A}) \geq d, \qquad \text{whenever} \quad d \geq 2 \Lambda \sqrt{A}.$$
Finally, for $d \geq \sqrt{C_1}$ we get  $\exp \big({- \frac{d^2}{C_1} \cdot (4^i-1)}\big)
\leq e^{-i}.$  Substituting this into \eqref{est_ak1} yields the inequality in the claim for  $d \geq \max\{ 2 \, \Lambda \sqrt{A}, \sqrt{C_1} \}$. This finishes the proof of the proposition. 
\end{proof}

 \section{Proof of the main results} \label{sec_main_proof}
\subsection{Ricci flow for almost non-negatively curved manifolds}

The main ingredient of the following proof is the use of the heat kernel estimates from Proposition \ref{thm:heat} to control the curvature growth.

\begin{proof}[{\bf Proof of Theorem \ref{thm:annco}.}]
Let $\mathcal{C} = \mathcal{C} (S, 0)$ and $\ell(p, t)$ be defined as in \eqref{eq_def_CC} and \eqref{def_ell} for $S = S_1$. If we rescale the metric by a large constant, we may assume that 
\[\ell(\cdot, 0) \leq \eps \leq \eps_0\]
for some small $\eps_0 = \eps_0(n, v_0) >0$ to be conveniently chosen later. 
Now consider the maximal time interval $[0, t_1)$ such that the Ricci flow $g(t)$ exists and satisfies 
\bec \lb{stop_rule}
\ell(\cdot, t) \leq 1 \qquad \text{and}  \qquad \vol_{g(t)}\big(B_{g(t)}(\cdot, 1)\big) \ge v_0/2 
\eec
for all $t \in [0, t_1)$. 
Note that this implies that $\Ric_{g(t)} \geq - (n-1)g(t)$ for all $t \in [0, t_1)$.
By maximum principle arguments and standard ODE estimates, it follows easily that $t_1 > 0$. The goal is to find a constant $t_0 = t_0(n, v_0) >0$ such that $t_1 \geq t_0$. First notice that if $t_1 > 1$ we are done; hence we suppose hereafter that $t_1 \leq 1$.

We first claim that there exists a constant $C_1=C_1(n,v_0) > 0$ such that 
\bec \lb{Rm_Ct}
|\Rm_{g(t)} | \le  \frac{C_1}{t}  \qquad \text{for all} \quad t\in (0,t_1).
\eec
\noindent To prove \eqref{Rm_Ct} we argue by contradiction: otherwise, we find a sequence of Ricci flows $(M_i, g_i(t))_{t \in [0, T_i)}$ with $T_i \leq 1$ satisfying \eqref{stop_rule} for all $t \in [0, T_i)$ such that
\[ Q_i := \sup_{t \in (0,T_i)} \sup_{M_i} \big( t \cdot |\Rm_{g_i(t)} | \big)  \xrightarrow{i \to \infty} \infty. \]
Pick $(p_i, t_i) \in M_i \times (0, T_i)$ such that $| \Rm_{g_i(t_i)} (p_i) | \geq \frac12 \frac{Q_i}{t_i}$.
So if we consider the parabolic rescaling $\tilde g_i(t) = \frac{Q_i}{t_i} g_i(t_i + t \cdot \frac{t_i}{Q_i})$, then $|\Rm|_{\tilde g_i(0)}(p_i) \geq 1/2$.
Moreover, $| \Rm_{\tilde g_i (t)} | \leq 2$ and $\Rm_{\tilde g_i(t)} + \frac1{Q_i} \I \in \mc C$ for $t \in (-Q_i /2, 0]$. As as $Q_i \to \infty$, Hamilton's compactness theorem implies that a subsequence of these flows, pointed at $(p_i,0)$, converges to a non-flat ancient solution $\tilde g_\infty(t)$ with bounded curvature and satisfying  $\Rm_{\tilde g_\infty(t)} \in \mc C$ for all times. Combining the lower volume bound in \eqref{stop_rule} with Bishop-Gromov comparison, we deduce that the asymptotic volume ratio of the limit solution is positive.
This contradicts the fact that the volume ratio on a $\kappa$-solution vanishes (see \cite[11.4]{P1}) --- and shows \eqref{Rm_Ct}.

Second, by Proposition \ref{lem:ell} there is a dimensional constant $C_2 > 0$ such that $\ell$ satisfies
\[ \partial_t \ell \leq \Delta \ell + \scal \ell + C_2 \ell^2 \leq \Delta \ell + \scal \ell + C_2 \ell \]
in the viscosity sense.
So by the maximum principle, $\ell (\cdot ,t) \leq e^{C_2 t} h$ on $M \times [0,t_1)$, where $h$ is the solution to the initial value problem 
$$\partial_t h = \Delta h+ \scal  h, \qquad h(0,\cdot)\equiv \eps.$$
We can express this solution as 
\[
h(x,t)= \eps \int_{M} G(x,t;y,0) \,d\mu_{0}(y),
\]
where $G(\cdot,\cdot\,;y,s)$ is the heat kernel defined in \eqref{G_conj}.  By \eqref{stop_rule} and volume comparison, we obtain that $\vol_{g(t)}\(B_{g(t)}(x,\sqrt{t})\) \geq c_0 v_0 t^{n/2}$ for some dimensional constant $c_0 > 0$.
Because of this and \eqref{Rm_Ct} we are now in position to apply Proposition \ref{thm:heat} to conclude for some constants $C_3 = C_3 (n, v_0) > 0$ and $C_4 = C_4 (n, v_0) > 0$
\begin{equation} \label{eq_ell_eps_bound}
\ell(x, t) \leq e^{C_2} h(x, t) \leq  \eps \cdot \frac{C_3}{t^{n/2}} \int_M \exp\Big({-\frac{d_0^2(x, y)}{C_3t} }\Big) \, d\mu_{0}(y) \leq   \eps C_4 \leq  \eps_0 C_4 \,
\end{equation}
for all $(x, t) \in M \times (0, t_1)$. A suitable choice of $\eps_0 = \eps_0(n, v_0) > 0$ ensures that $\ell \leq 1/2$ on $M \times [0, t_1)$.
This bound and \eqref{Rm_Ct} imply that the curvature bound in \eqref{stop_rule} holds on a time-interval that is larger than $[0,t_1)$.

Lastly, we turn our attention to the volume condition in \eqref{stop_rule}. Our curvature condition in \eqref{stop_rule} combined with \eqref{Rm_Ct} and Hamilton's trick give us the appropriate double side control on $d_{g(t)}$ in terms of $d_{g(0)}$, which arguing exactly as in \cite[Corollary 6.2]{Simon12}, implies the existence of a time $\tau = \tau (n, v_0) > 0$ such that
$$\vol_{g(t)}(B_{g(t)}(p, 1)) \geq 2 v_0/3 \qquad \text{for all} \quad t \in [0, \min \{ \tau, t_1 \}).$$
This bound and \eqref{Rm_Ct} imply that the volume condition of \eqref{stop_rule} holds on a time-interval that is larger than $[0,t_1)$, unless $t_1 \geq \tau$.

Combining the results of the previous two paragraphs, we obtain by the maximal choice of $t_1$ that $t_1 \geq \tau$.
The theorem now follows from \eqref{Rm_Ct} and \eqref{eq_ell_eps_bound} by undoing the initial parabolic rescaling.
\end{proof}

\subsection{Volume growth of weakly $\bm{\text{PIC}_1}$ ancient solutions}
In this subsection, we generalize Perelman's analysis of $\kappa$-solutions (see \cite[11.4]{P1}), which have non-negative curvature operator, to the weakly $\text{PIC}_1$-case.
These generalizations will only be used in the almost $\text{PIC}_1$ case and the almost 2-nonnegative curvature cases of Theorem~\ref{thm:ann_gral}.

\begin{lem} \label{lem_v_decomposition}
For any $v \in \so (n, \C)$ of rank $2$ and with eigenvalues of norm $\alpha$ there are
 $u, w \in \so (n, \C)$ with $v = u + w$ such that $|u| = \alpha$ and such that every linear combination $u + sw$, $s \in \R$ has rank $2$ and eigenvalues of norm $\alpha$.
\end{lem}

\begin{proof}
If $\alpha = 0$, then we automatically have $v^3 = 0$.
So we can set $u = v$ and $w = 0$.

Assume now that $\alpha > 0$.
By multiplying $v$ with an appropriate complex number of norm $\alpha^{-1}$, we can reduce the lemma to the case in which $v$ has eigenvalues $\pm 1$.
Let $x^\pm = x^\pm_1 + i x^\pm_2 \in \C^n$ be the corresponding eigenvectors.
By conjugating $v$ with a real-valued orthogonal matrix, we can assume without loss of generality that $x^+_1, x^+_2 \in \spann \{ e_1, e_2 \}$ and $x^-_1, x^-_2 \in \spann \{ e_1, \ldots, e_4 \}$, where $e_1, \ldots, e_n$ denote the standard basis vectors of $\C^n$.
So $v$ takes the form
\[
v = \begin{pmatrix}
  \begin{matrix}
  0 & v_{12} & v_{13} & v_{14}  \\-v_{12} & 0 & v_{23} & v_{24}  \\-v_{13} & -v_{23} & 0 & v_{34}  \\-v_{14} & -v_{24} & -v_{34} & 0
  \end{matrix}
  & \rvline &  \bigzero  \\
\hline 
 \parbox[c]{0.4cm}{\bigzero} & \rvline &
  \begin{matrix}
  \parbox[c]{0.4cm}{\bigzero} 
  \end{matrix}
\end{pmatrix}
\]

Due to the eigenvalue equation $v x^+ = x^+$ we must have $v_{12} = \pm i$, $v_{13} = \pm i v_{23}$ and $v_{14} = \pm i v_{24}$.
So there is a non-trivial linear combination of the first two columns that only has non-zero entries in the first two coordinates.
As $\rank v = 2$, this implies that $v_{34} = 0$.

Now let $u$ be the matrix that agrees with $v$ on the upper $2\times 2$-block and has zero entries everywhere else and set $w := v - u$.
It is not hard to verify the $u$ and $w$ have the desired properties.
\end{proof}

We can now prove the main result of this subsection.

 \begin{lem}\label{lem:ancient} Let $(M^n,g(t))_{t\in (-\infty,0]}$ be a nonflat ancient solution of the Ricci flow with bounded curvature satisfying weakly $\emph{PIC}_1$. Then it 
 has nonnegative complex sectional curvature. 
 Furthermore, the volume growth
 is non-Euclidean, i.e. $\ds \lim_{r \to \infty} r^{-n} \vol B_{g(0)} (x, r) = 0$ for all $x \in M$.
 \end{lem}
 \begin{proof} 
 For all $t \geq 0$ we define 
  \[
 C(t):=\bigg\{\Rm\in S^2_{B}(\so(n)) \;\Big|\;  \Rm (v,\bar v)\ge -1 \begin{array}{l}\mbox{ for all $v \in \so(n,\C)$ with $\rank v = 2$ }\\\mbox{ and eigenvalues of norm $\leq \sqrt{2t}$}\end{array} \hspace{-2mm}\bigg\}
 \]
 As explained in  \cite[sec~4]{WLie}, the family $(C(t))_{t\in[0,\infty)}$ is continuous in $t$ and $C(0)$ is the set of weakly $\text{PIC}_1$ curvature operators.
 Note that $(C(t))_{t\in[0,\infty)}$ is monotone in the sense that $C (t_2) \subset C(t_1)$ whenever $t_1 \leq t_2$.
 We claim that $(C(t))_{t\in[0,\infty)}$ is invariant under the ODE: $\Rm'=2 Q(\Rm)$.
 That is if  $\Rm(t)$ is a solution to this ODE with $\Rm(t_0)\in C(t_0)$, then $\Rm(t)\in C(t)$ for all $t\ge t_0$.
 To prove our claim, it suffices to consider the case in which $\Rm (t_0)$ lies in the interior of $C(t_0)$.
 
We then define 
\begin{equation} \label{eq_lambda_inf}
 \lambda(t) := \inf \left\{\Rm(t) (v, \bar v) \;\Big|\;  \begin{array}{l}\mbox{$v \in \so(n,\C)$ with $\rank v = 2$ }\\\mbox{and eigenvalues of norm $\leq 1$}\end{array} \hspace{-2mm}\right\} 
\end{equation}
Note that $\Rm (t) \in C(t)$ if and only if $\lambda(t) \geq - \frac1{2t}$.
Let us now derive a differential inequality for $\lambda (t)$.

Fix $t > t_0$ for a moment and assume that $\Rm := \Rm (t)$ is still contained in the interior of $C(t_0)$.
So $\Rm$ is also contained in the interior of $C(0)$ and therefore $\Rm (v, \bar v) > 0$ for all non-zero nilpotent rank 2 matrices $v \in \so (n, \C)$ with eigenvalues zero.
We now claim that the infimum in (\ref{eq_lambda_inf}) is attained.
To see this, consider a minimizing sequence $v_i \in \so (n, \C)$ with $\rank v_i = 2$ and eigenvalues of norm $\leq 1$ such that $\lim_{i \to \infty} \Rm (v_i, \bar v_i ) = \lambda (t)$.
If $|v_i|$ remains bounded, then by compactness, we can pass to a subsequential limit $v_\infty \in \so (n, \C)$ with $\Rm (v_\infty, \bar v_\infty) = \lambda (t)$.
Assume now that $|v_i| \to \infty$ and let $v'_\infty$ be a subsequential limit of $v_i / |v_i|$.
Then $v'_\infty$ has rank 2,  eigenvalues 0 and by the choice of the $v_i$ we have $\Rm ( v'_\infty, \bar v'_\infty ) = 0$, contradicting our previous conclusion.

So we can choose $v \in \so (n, \C)$ as in (\ref{eq_lambda_inf}) with $\lambda (t) = \Rm (v, \bar v)$ and consider a splitting $v = u + w$ as in Lemma \ref{lem_v_decomposition}.
By the choice of $v$ the function $s \mapsto \Rm ( u + sw, \bar u + s \bar w)$ attains a minimum at $s =1$, which implies
\[ \Rm (u, \bar w) + \Rm (w, \bar u ) + 2 \Rm (w, \bar w ) = 0. \]
It follows that
\[ \lambda (t) = \Rm (u , \bar u ) + \tfrac12 \Rm (u, \bar w) + \tfrac12 \Rm (w, \bar u) = \Re \big( \Rm (v, \bar u ) \big). \]
Now $|u| \leq 1$ gives $| \Rm (v) |^2 \geq | \Rm (v, \bar u )|^2 \geq \lambda^2 (t)$.
Furthermore, it follows from the proof of \cite[Theorem~1]{WLie} that $\Rm^\# (v, \bar v) \geq 0$.
So we obtain that in the barrier sense
\[ \frac{d}{dt^-} \lambda(t) \geq 2 | \Rm (v) |^2 + 2 \Rm^\# (v, \bar v ) \geq 2 \lambda^2(t). \]
By ODE-comparison, this proves the invariance of $C(t)$ under the ODE:  $\Rm' = 2 Q(\Rm)$.
 
Since $(M,g(t))$ has bounded curvature, we can apply a dynamical version of the maximum principle (cf.~\cite[12.37]{Chow2}) to $(M,g(t-t_0))_{t\in [0, t_0]} $ for any $t_0 > 0$.
Since $\Rm_{g(-t_0)} \in C(0)$, we have $\Rm_{g(0)} \in C(t_0)$.
Letting $t_0 \to \infty$ implies that $(M,g(0))$ has nonnegative complex curvature. 
By shifting the flow we get the same the result for each metric $g(t)$. 

The claim about the volume growth now follows from \cite[Lemma 4.5]{CRW}.
 \end{proof}

\subsection{Proofs of the remaining global results}

\begin{proof}[{\bf Proof of Theorem \ref{thm:ann_gral}.}]
The proof is almost exactly the same as the proof of Theorem~\ref{thm:annco}.
In the Riemannian case, we define $\ell(p,t)$ as in \eqref{def_ell} for the appropriate $S_i$, $i = 2, 4, 5$.
In each case the bound $\ell \leq 1$ implies a lower bound on the Ricci curvature.
So all computations from the previous proof carry over.
To obtain (\ref{Rm_Ct}) observe that each curvature condition implies weakly $\text{PIC}_1$, and therefore we get a contradiction by using Lemma~\ref{lem:ancient} instead of \cite[11.4]{P1}.

 To prove the K\"ahler case one simply follows the same arguments but using $\tilde \ell$ from \eqref{def_tilde_ell} instead of $\ell$ and Proposition~\ref{lem:ell_Kaehler} instead of Proposition~\ref{lem:ell}. For complete K\"ahler manifolds with bounded curvature, recall that short-time existence is guaranteed by \cite[Theorem 5.1]{Shi}, and hence one can find as before a starting time interval $[0, t_1)$ where \eqref{stop_rule} holds. Then the only difference is that in order to get the contradiction at the end of the proof of \eqref{Rm_Ct}, we need to use \cite[Theorem~2]{Ni} instead of Lemma \ref{lem:ancient}.
\end{proof}

\begin{proof}[{\bf Proof of Corollary \ref{anco_nnc}}]
We argue by contradiction: If the statement was false, then we can find a sequence of counterexamples, that is, a sequence of closed Riemannian $n$-dimensional manifolds $\{(M_i, g_i)\}_{i \in \mathbb N}$ satisfying 
\bec \lb{counter_cond}
\vol_{g_i}(M_i) \geq v_0, \qquad \Rm_{g_i} + \frac1{i} \I \in \mc C\qquad \text{and} \qquad  {\rm diam}_{g_i}(M_i) \leq D.
\eec
so that each $M_i$ admits no metric $\bar g_i$ with $\Rm_{\bar g_i} \in \mc C$.

 The volume and curvature conditions in \eqref{counter_cond} allow us to use Theorems \ref{thm:annco} or \ref{thm:ann_gral} to deduce that there exists a sequence of Ricci flows $\{(M_i, g_i(t))\}_{t \in [0, \tau]}$ with curvature control $|\Rm |_{g_i(t)} \leq \frac{C}{t}$ for positive times, where $\tau$ and $C$ are independent of $i$. The latter curvature control and the volume bound in \eqref{counter_cond} yield, by means of Hamilton's compactness theorem, a limiting Ricci flow $(M_\infty, g_\infty(t))_{t \in (0, \tau]}$. By definition of Cheeger-Gromov convergence the curvature condition in \eqref{counter_cond} passes to the limit, and hence $\Rm_{g_\infty(t)} \in \mc C$. Finally the uniform diameter bound in \eqref{counter_cond} ensures that $M_\infty$ is diffeomorphic to $M_i$ for all $i$ large enough. 
This, however, contradicts our choice of $M_i$. The same proof works in the K\"ahler case.
\end{proof}

\begin{proof}[{\bf Proof of Corollary \ref{anco_smooth}}]
Arguing as in the previous proof, we get a limiting Ricci flow $(M_\infty, g(t))_{t\in (0, \tau]}$ with $\Rm_{g(t)} + \eps_\infty \I \in \mc C$ for all $t > 0$. The proof of \eqref{init_sing} follows as in \cite[Theorem 7.2]{Simon09} by applying twice the triangle inequality of the Gromov-Hausdorff distance combined with a two-sided distance distortion control on $d_{g(t)}$.

Lastly, let us show that $X$ is homeomorphic to $M_\infty$.
As the Ricci curvature of $g(t)$ is uniformly bounded from below and we have a bound of the form $|{\Rm_{g(t)}}| \leq C/t$ for some generic $C <\infty$, due to (\ref{curv_bdd_thm2}), we have a distance distortion bound of the form
\[ e^{-C(t_2 -t_1)} d_{g(t_2)} (x,y)  \leq d_{g(t_1)} (x,y)  \leq  d_{g(t_2)} (x,y)+ C \sqrt{t_2},  \]
for any $x, y \in M_\infty$ and $0 < t_1 < t_2$.
So the limit
\[ d_0 (x,y) := \lim_{t \searrow 0} d_{g(t)}(x,y) \]
exists and for all $t > 0$
\begin{equation} \label{eq_d0_dgt}
 e^{-Ct} d_{g(t)} (x,y) \leq d_0 (x,y) \leq d_{g(t)} (x,y) + C \sqrt{t}. 
\end{equation}
Hence $(M_\infty, d_0)$ is a metric space, which due to (\ref{init_sing}) is isometric to $(X, d_X)$.
Fix some $t > 0$.
By the first inequality in (\ref{eq_d0_dgt}), the identity map $(M_\infty, d_0) \to (M_\infty, d_{g(t)})$ is continuous. 
To see that its inverse is continuous, consider a convergent sequence $x_i \to x_\infty$ in $(M_\infty, d_{g(t)})$.
Fix some $\eps > 0$ and choose $t' := \(\frac{\eps}{2 C}\)^2$.
For large $i$ we have $d_{g(t')} (x_i, x_\infty) < \eps / 2$.
Thus by the second inequality in (\ref{eq_d0_dgt}), we have $d_0 (x_i, x_\infty) < \eps$.
As $\eps$ was chosen arbitrarily, this shows that $x_i \to x_\infty$ in $(M_\infty, d_0)$.
So the identity map $(X, d_X) \cong (M_\infty, d_0) \to (M_\infty, d_{g(t)})$ is a homeomorphism, which finishes the proof.
\end{proof}

\section{Proof of the local statement} \label{sec_loc_Thm}
We will now prove Theorem~\ref{anco_local}.
The idea is to reduce the proof to an application of Theorem \ref{thm:ann_gral} after a modification of the original metric by a conformal change that pushes the boundary of the relevant region, on which we have curvature bounds, to infinity in such a way
that the modified metric is complete and has bounded curvature. The conformal factor we use is a modification of that for the hyperbolic metric on a unit Euclidean ball. 

\begin{proof}[{\bf Proof of Theorem~\ref{anco_local}.}]
After rescaling, we can assume without loss of generality that $r = 1$. 
Hereafter we use the notation $B_R := B_R(U)$ for the $R$-tubular neighborhood around $U$.
Our curvature assumption on $B_1$ implies that $\sec \geq - \eps$ and therefore we can apply the Hessian comparison theorem to the Riemannian distance function $d_p$ to any point $p\in U$. 
We obtain that $d_p$ satisfies on $B_1$
$$\nabla^2 d^2_p \leq 2\sqrt{\eps} \coth(\sqrt{\eps} d_p) d_p \cdot g \leq C_0 g . $$
At the points where $d^2_p$ is not smooth, the first inequality is understood in the barrier sense.  Then $d_U:= \inf\{d_p(\cdot) \, |\, p \in U \}$ also satisfies $\nabla^2 d_U^2 \leq C_0 g$ on $B_1$ in the barrier sense. 
 Using the approximation technique of Greene-Wu (see 
\cite[p.~644]{GW1}, \cite[p.~60]{GW2},  and \cite[Lemma 8]{GW3}) we can construct a smooth function $\rho : B_{1 - 1 / 8} \fle \R$ such that
\begin{enumerate}[label=(\alph*)]
\item $|\rho - d^2_U | < 1 / 16$.
\item $|\nabla \rho | < 4$.
\item $\nabla^2 \rho < 2C_0  g$.
\end{enumerate}
By (a) the sublevel set $\{ \rho < 1 - 1 / 8 \}$ is disjoint from $\partial B_{1 - 1 / 32}$ and contains $U$.
Let $V$ be the connected component of this sublevel set that contains $U$.

Next we will perform a conformal change of the metric by pushing the boundary $\partial V$ to infinity, but keeping some bounds on the curvature operator and the volume. 
With this goal in mind, let $\varphi: [0, \infty) \to [0, 1]$ be a smooth cut-off function with the following properties:
\begin{enumerate}[label=(\roman*)]
\item $\varphi(s) = 0$ for $s \leq 1- 1/4$.
\item $\varphi(s) > 0$ for $s > 1 - 1 / 4$.
\item $\varphi(s) = 1$ for $s \geq 1- 1/8$.
\item $\varphi' \geq 0$ and $|\varphi'|, |\varphi''|\leq C$.
\end{enumerate}
We now consider the following conformal change of our original metric on $V$:
\[\hat g = \Phi^2 g, \qquad \text{where} \qquad \Phi := \frac1{1 - (\varphi \circ \rho)}.\]
By (a) above, we have $B_{1- 1 / 2} \subset \{ \rho < 1- 1 / 4 \} \cap V$, and thus $\hat g \equiv g$ on $B_{1 - 1/2}$. 
Moreover, $\Phi^{-1}$ goes to zero near $\partial V$.
So 
In order to show that $(V, \hat g)$ is complete, consider a smooth curve $\gamma : [0,l) \to V$ parameterized by arclength with respect to $\hat g$.
This implies that $\Phi (\gamma(s)) |\gamma'(s)| = 1$ and thus $|\gamma'(s)| = \Phi^{-1} (\gamma(s)) \leq 4 C d_g ( \gamma(s), \partial V)$ by the fact that $\Phi^{-1}$ is $4 C$-Lipschitz.
Therefore, if $l < \infty$, then $\lim_{s \to l} \gamma(s) \in V$.

Next, set $f:=  \log(\Phi)$ and note that the curvature operator of $\hat g$, viewed as a $(0,4)$-tensor, can be expressed as
\begin{equation} \label{eq_conformal_change_formula}
\Rm_{\hat g} = e^{2 f} \big(\Rm_g - A \owedge g \big), \quad \text{for } \quad A:= \nabla^2  f - df \otimes df + \frac1{2} |\nabla f|^2 g.
\end{equation}
Fix some point $x \in V$ and let $\{e_i\}_{i = 1}^n$ be an orthonormal basis for $(T_x M,  g_x)$ that diagonalizes $A$ with eigenvalues $a_i = A(e_i, e_i)$.
Then for any $v =  v_{ij} e_i \wedge e_j$ with $|v|_{\hat g} = \Phi^2 |v |_g  \leq 1$ we compute, using the Kulkarni-Nomizu product $\owedge$ as defined in (\ref{def_KN}),
 \[ \big( \Rm_{\hat g} - \Phi^{2} \Rm_g \big) (v, v)  = - \Phi^{2}(A \owedge g) (v,v) 
 = -4 \Phi^{2}  a_i \, v_{ij} v_{ij} \geq - 4 \Phi^{-2} \max_{1 \leq i \leq n} a_i.\]

Using that 
$$d f = \Phi \cdot (\varphi' \circ \rho) \, d \rho \qquad  \nabla^2 f = \Phi \cdot (\varphi' \circ \rho) \, \nabla^2 \rho  + \big( \Phi \cdot ( \varphi'' \circ \rho ) + \Phi^2 \cdot (\varphi ' \circ \rho)^2 \big) \, d\rho \otimes d \rho,$$
and, taking into account that $\Phi \geq 1$ on $V$, we can estimate
\[
\Phi^{-2} \, a_i \leq \Phi^{-2} \big(  \nabla^2 f(e_i, e_i) + |\nabla f|_g^2 \big) \leq C(n),
\]
where we have used condition (iv) in the definition of $\varphi$ and property (c).
It follows that $\Rm_{\hat g} - \Phi^{2} \Rm_g + C(n) \I_{\hat g}$ is non-negative definite.
Hence, since $\I_{\hat g} = \Phi^4 \I_{g}$ and $\Phi \geq 1$,
\begin{multline*}
 \Rm_{\hat g} + (C(n) + \eps) \cdot \I_{\hat g}  = \eps \Phi^2 (\Phi^2 - 1 ) \cdot \I_{ g} + \Phi^{2} (\Rm_g + \eps  \I_{ g}) \\
 + \big(\Rm_{\hat g} - \Phi^{2} \Rm_g + C(n) \cdot \I_{\hat g} \big) \in \mc C. 
\end{multline*}
So if we define $\hat{\ell}$ as in (\ref{def_ell}) for $\hat g$ instead of $g$, then
\[ \hat\ell \leq \eps + C(n) . \]

As $\overline{V}$ is compact, we obtain that $e^{-2f} |\nabla^2 f|_g \leq \Phi^{-1} |\nabla^2 \Phi^{-1}|_g +  |\nabla \Phi^{-1}|_g^2$ and $e^{-2f} |\nabla f|^2_g \linebreak[1] \leq \linebreak[1]  |\nabla \Phi^{-1}|^2_g$ are uniformly bounded on $V$.
So by (\ref{eq_conformal_change_formula}) we know that $\Rm_{\hat g}$ is uniformly bounded.

Next we claim that there exists a constant $ \hat v_0 = \hat v_0(n, v_0) > 0$ such that 
\[\vol_{\hat g} \big(B_{\hat g}(x, 1)\big)\geq \hat v_0 > 0 \qquad \text{for all} \quad x \in V.\]
To see this fix some $x \in V$ and observe that by properties (b) and (iv), $\Phi^{-1}$ is $4 C$-Lipschitz and thus
\[ \frac12 \Phi (x) \leq \Phi \leq 2 \Phi (x) \quad \text{on} \quad B_{g}(x, r_{x}) \qquad \text{with} \quad r_{x}:= \frac{\Phi^{-1}(x)}{8C}. \]
Therefore $B_{\hat g} (x, 1) \supset B_g (x, r_{x})$
and thus, by a volume comparison estimate in $B_g (p,1)$ with $p \in U$, we obtain
\[ \vol_{\hat g} \big( B_{\hat g} (x, 1) \big) \geq \big(\tfrac12\Phi(x) \big)^{n} \vol_g \big( B_g (x, r_{x}) \big) \geq c(n) v_0 . \]

In summary, we have constructed a complete Riemannian manifold $( V, \hat g)$ with bounded curvature that satisfies the assumptions of Theorem~\ref{thm:ann_gral}.
Hence we can find $t_0 = t_0(n, v_0), C_0 = C_0 (n, v_0) > 0$,  and a Ricci flow $\hat g(t)$ with $\hat g(0) = \hat g$ satisfying
$$\hat\ell(\cdot, t) \leq  C_0(\eps + C(n)) \quad \text{and} \quad |\Rm_{\hat g(t)} | \leq \frac{C_0}{t} \quad \text{for all} \quad t \in (0, t_0].$$

To improve the lower bound on $\hat\ell$ on $U$, we argue as in equation (\ref{eq_ell_eps_bound}), in the proof of Theorem \ref{thm:ann_gral}:
For any $x \in U$ we have for some generic constant $C_1 = C_1 (n, \hat v_0)> 0$
\begin{align*}
 \hat\ell (x,t) &\leq \frac{C_1}{t^{n/2}} \int_{V} \exp \bigg({ - \frac{d_{\hat g}^2 (x,\cdot \, )}{C_1 t}}\bigg) \hat\ell (\cdot\, , 0) d\mu_{\hat g} \\
 &\leq \frac{C_1}{t^{n/2}} \bigg[\int_{B_{1- 1 / 2}} \!\!\!  \!\!\!\exp\bigg({ - \frac{d_{\hat g}^2 (x, \cdot)}{C_1 t}}\bigg) \ell(\cdot , 0)  d\mu_{\hat g} + \int_{V \setminus B_{1-1 / 2}}  \!\!\! \!\!\!  \exp \bigg({ - \frac{d_{\hat g}^2 (x,\cdot\, )}{C_1 t}}\bigg) \hat\ell (\cdot\, , 0) d\mu_{\hat g}\bigg] \\
& \leq C_1 \eps +  (C(n) + \eps) \cdot \frac{C_1}{t^{n/2}} \int_{V \setminus B_{1-1 / 2}} \exp \bigg({ - \frac{d_{\hat g}^2 (x,\, \cdot)}{C_1 t}}\bigg)  d\mu_{\hat g}
\end{align*}
If $t \leq t_1 (n, \eps)$, then the second term can be bounded by $C_1 \eps$.

The theorem now follows by restricting $\hat g(t)$ to $U$.
\end{proof}

We are finally in position to get a short-time existence result without upper curvature bounds.
\begin{proof}[{\bf Proof of Theorem~\ref{ste_new}.}]
The idea is to apply use the construction of the proof of Theorem~\ref{anco_local} with $r = 1$ and $U_i = B_g(p, R_i)$, for a sequence of radii $R_i \to \infty$. 
For each $i$, we obtain a Ricci flow $(V_i, g_i(t))_{t \in [0, \tau)}$ with bounded curvature defined on a complete manifold $V_i \supset U_i$, with $\tau = \tau(n, v_0)$, satisfying $g_i(0)|_{U_i} = g$ and so that $g_i(t)$ satisfies the desired curvature bounds \eqref{curv_bdd_thm2} for uniform constants. 
By Perelman's Pseudolocality Theorem (cf.~\cite[Theorem 10.3]{P1}) we obtain local uniform curvature bounds on an open neighborhood of $M \times \{0 \}$ in $M \times [0, \tau)$.
Combining these bounds with \eqref{curv_bdd_thm2} yield local uniform curvature bounds on $M \times [0, \tau)$.
Hence, by the generalized interior estimates of Shi (see \cite[Theorem 11]{LuTian}), we can pass to the limit to get a complete 
Ricci flow $(M, g(t))_{t \in [0, \tau)}$ that satisfies the curvature bounds in \eqref{curv_bdd_thm2}. 
\end{proof}

\end{document}